\documentclass[11pt]{amsart}

\usepackage{latexsym,amsfonts,amssymb,epsfig,tabularx,amsthm,dsfont,enumerate}
\usepackage{amsmath}
\theoremstyle{definition}

\numberwithin{equation}{section}

\newtheorem{thm}{Theorem}[section]
\newtheorem{rem}[thm]{Remark}
\newtheorem{lem}[thm]{Lemma}
\newtheorem{prop}[thm]{Proposition}
\newtheorem{cor}[thm]{Corollary}

\newtheorem{defi}[thm]{Definition}

\newcommand{\cf}{\ensuremath{\mathcal F}}

\newcommand{\tor}{\ensuremath{\mathbb{T}^d}}
\newcommand{\Z}{\mathbb{Z}}
\newcommand{\N}{\mathbb{N}}
 
\newcommand{\R}{\mathbb{R}}
\newcommand{\C}{\mathbb{C}}

\newcommand{\abs}[1]{\Big\vert #1 \Big\vert}	% absolute value / norm
\newcommand{\ls}{\langle}
\newcommand{\rs}{\rangle}
\renewcommand{\l}{\Big\langle}
\renewcommand{\r}{\Big\rangle}
\renewcommand{\a}{\alpha}

\renewcommand{\phi}{\varphi}
\newcommand{\e}{{\rm e}}
\newcommand{\eps}{\varepsilon}

\newcommand{\wt}{\widetilde}
\newcommand\dint{{\,\rm d}}
\newcommand{\supp}{{\rm supp}}
\newcommand{\sinc}{{\rm sinc}}

\newcommand{\X}{\mathbb{X}}

\renewcommand{\S}{\mathcal{S}}
\newcommand{\F}{\mathcal{F}}

\newcommand{\Bspt}{\ensuremath{\mathbf{B}^s_{p,\theta}}}
\newcommand{\Fspt}{\ensuremath{\mathbf{F}^s_{p,\theta}}}
\newcommand{\Fsptw}{\ensuremath{\mathbf{F}^s_{p,2}}}
\newcommand{\Wsp}{\ensuremath{\mathbf{W}^s_p}}

\newcommand{\Bo}{\ensuremath{\mathring{\mathbf B}_{p,\theta}^s}}
\newcommand{\Fo}{\ensuremath{\mathring{\mathbf F}_{p,\theta}^s}}
\newcommand{\Wo}{\ensuremath{\mathring{\mathbf W}_p^s}}

\newcommand{\A}{\ensuremath{\mathbf{A}}}
\newcommand{\B}{\ensuremath{\mathbf{B}}}
\newcommand{\T}{\ensuremath{\mathbf{F}}}

\newcommand{\Aspt}{\ensuremath{\mathbf{A}^s_{p,\theta}}}
\newcommand{\Ao}{\ensuremath{\mathring{\mathbf A}_{p,\theta}^s}}

\newcommand{\bproof}{\begin{proof}}
\newcommand{\eproof}{\end{proof}}

\title[Frolov's cubature]{The role of Frolov's cubature formula for functions
with bounded mixed derivative}   


\date{\today}

\author{ 
Mario Ullrich
\and 
Tino Ullrich
}

\begin{document}
 
\begin{abstract} 
We prove upper bounds on the order of convergence of Frolov's cubature formula 
for numerical integration in function spaces of dominating mixed smoothness 
on the unit cube with homogeneous boundary condition. 
More precisely, we study worst-case integration errors for Besov $\Bspt$ and Triebel-Lizorkin spaces $\Fspt$ and our results 
treat the whole range of admissible parameters 
%i.e.~values of $s$, $p$ and $\theta$ such that the spaces consist of continuous functions.
$(s\geq 1/p)$. 
In particular, we obtain upper bounds for the difficult the case of small smoothness which is 
given for Triebel-Lizorkin spaces $\Fspt$ in case $1<\theta<p<\infty$ with $1/p<s\leq 1/\theta$. 
The presented upper bounds on the worst-case error show a completely different behavior 
compared to ``large'' smoothness $s>1/\theta$. 
In the latter case the presented upper bounds are optimal, i.e., they can not be improved 
by any other cubature formula. The optimality for ``small'' smoothness is open. 
\end{abstract} 

\maketitle

\section{Introduction}
The efficient integration of multivariate functions is a
crucial task for the numerical treatment of many multi-parameter real-world
problems. The computation of the integral can
almost never be done analytically since often the available
information of the signal or function $f$ is highly incomplete or simply no closed-form solution exists. 
A cubature rule approximates the integral $I(f) = \int_{[0,1]^d} f(x) \, dx$ 
by computing a weighted sum of finitely many function values and 
the $d$-variate function $f$ is assumed to belong to some (quasi-)normed function space 
$\mathbf{F}_d \subset C([0,1]^d)$ of continuous functions.
The optimal worst-case error with respect to 
$\mathbf{F}_d$ is given by 
%\begin{equation}\label{f}
  %\mbox{Int}_N(\mathbf{F}_d):=\inf\limits_{\substack{X_N \subset [0,1]^d\\ \# X_N = N}}\sup\limits_{\|f\|_{\mathbf{F}_d} \leq 1}
%|I(f)-I_N(X_N,f)|\,.
%\end{equation}
\begin{equation}\label{eq:minimal}
  \mbox{Int}_n(\mathbf{F}_d) \,:=\, \inf\limits_{x^1,\dots,x^n\in[0,1]^d}\,
	\inf\limits_{\lambda^1,\dots,\lambda^n\in\R}\,
	\sup\limits_{\|f\|_{\mathbf{F}_d} \leq 1}
	\left|I(f)-\sum\limits_{i=1}^n \lambda^i f(x^i)\right|\,.
\end{equation}
In this paper we give  an (almost) complete answer for the question on the correct asymptotical behavior 
of $\mbox{Int}_n(\mathbf{F}_d)$ for several classes of functions with dominating mixed 
smoothness. The by now classical research topic of numerically integrating 
$d$-variate functions with mixed smoothness properties goes back to the work of 
Korobov~\cite{Ko59}, Hlawka~\cite{Hl62}, and Bakhvalov~\cite{Ba63} in the 1960s and was 
continued later by 
numerous authors including Frolov \cite{Fr76}, Temlyakov~\cite{Te86,Te90,Te91,Te03}, 
Dubinin~\cite{Du, Du2}, Skriganov~\cite{Sk94}, Triebel~\cite{Tr10}, 
Hinrichs et al.\ ~\cite{Hi10,HiMaOeUl14, HO14}, Hinrichs, Novak, M. Ullrich, Wo{\'z}niakowski
\cite{HNUW14a,HNUW14}, Hinrichs, Novak, M. Ullrich \cite{HNU14}, D\~ung, T. Ullrich \cite{DU14}, T. Ullrich
\cite{TU14}, 
Dick and Pillichshammer~\cite{DP}, and Markhasin~\cite{Ma13,Ma13a,Ma13b} to mention just a few. 
In contrast to the quadrature of univariate functions, where equidistant point
grids lead to optimal formulas, the multivariate problem is much more involved.
In fact, the choice of proper sets $X_n \subset [0,1]^d$ of integration nodes in
the $d$-dimensional unit cube is the essence of ``discrepancy theory'' and connected with deep problems
in number theory, already for $d=2$. 

We study Besov-Triebel-Lizorkin classes $\Bo$ and $\Fo$ with dominating mixed smoothness on the unit cube $[0,1]^d$ 
with homogeneous boundary condition and 
provide lower and upper bounds for $\mbox{Int}_n(\mathring{\mathbf{A}}^s_{p,\theta})$
which are sharp in order. Here and in the sequel $\A$ stands for $\B$ or $\T$.
This represents the crucial step for getting the same order of convergence also for
the larger periodic spaces $\Aspt(\tor)$ and non-periodic spaces $\Aspt(\R^d)$ 
as 
%the proposed modifications in
the recent paper \cite{NUU16} shows.
As a motivation let us emphasize that Besov regularity is the correct 
framework when it comes to 
so-called kink functions, which often occur in mathematical finance; e.g.,\ the pricing of a European call option, 
whose pay-off function possesses a kink at the strike price~\cite{glasserman2004monte}. 
Indeed, the simple example $f(t) = |2t-1|$ belongs to 
$\B^{2}_{1,\infty}(\mathbb{T})$, but not to any of the Sobolev spaces 
$\mathbf{W}^2_p(\mathbb{T})$, $p\ge1$.
%In a sense, one sacrifices integrability for gaining regularity. 
%Looking at the bounds \eqref{thm:error-B} and \eqref{thm:error-F} below we 
%see that certain cubature rules benefit from higher Besov regularity while 
%the integrability $p$ does not enter the picture (however the fine index $\theta$ does). Note, that the above
%example of a (tensorized) kink function does not belong to $\mathbf{W}^2_1$. 
%Therefore, it makes sense to study the classes $\Bspt$ especially for $p=1$.
Looking at the bounds \eqref{thm:error-B} and \eqref{thm:error-F} below we see that the main 
order of convergence only depends on the smoothness $s$. Hence, to explain the (optimal) order of convergence
$n^{-2}$ for the numerical integration of the function $f$, the classes $\Bspt$ represent an appropriate framework.
The method of our proof allows us to tackle also the (technically more difficult) 
Triebel-Lizorkin spaces $\Fo$, which are, in the special case 
$\theta=2$, $1<p<\infty$, the Sobolev spaces $\Wo$. Thereby, we obtain the first non-trivial upper bounds on 
$\mbox{Int}_n(\Wo)$ in the range $p>2$, $1/p<s\le1/2$, if the dimension is larger than 2.

We show that there exists a cubature formula, namely the classical Frolov construction, which provides the
optimal order of convergence for ``almost'' the entire Besov-Triebel-Lizorkin scale of 
function spaces (including anisotropic mixed smoothness, see Remark \ref{rem:universal} below). 
Note that in some cases we do not have matching lower bounds, cf.~Theorem~\ref{thm:error2}, however we 
conjecture the optimality also in those cases. Let us emphasize that,
in contrast to well-known sparse grid approaches, the nodes and weights of this cubature formulas are constructed 
independently of the parameters of the respective spaces; i.e., independently of $s$, $p$ 
and $\theta$. In other words, we do not need to incorporate any further information on the function to be integrated.

Frolov \cite{Fr76} introduced this method in 1976, 
see also~\cite{MU14} 
for a tutorial paper on the method. 
For the definition of the method, 
let $T\in\R^{d\times d}$ be a suitable matrix with unit determinant, 
see Section~\ref{sec:frolov}, and define the lattices $\X_n=n^{-1/d}T(\Z^d)$, 
$n>1$. Frolov's cubature formula is then rather simply defined as
\begin{equation}\label{eq:frolov}
Q_n(f) \,=\, \frac{1}{n}\sum_{x\in \X_n\cap[0,1)^d} f(x)
\end{equation}
and it is well known that $|\X_n\cap[0,1)^d|\asymp n$. For the detailed definition of the cubature formula $Q_n$ and the spaces $\Aspt$ we refer 
to Section~\ref{sec:frolov} and Section~\ref{sec:spaces}, respectively. 
We define the \emph{worst-case error} of the cubature rule $Q_n$ in a class $\mathbf{F}_d$ by 
\[
e(Q_n, \mathbf{F}_d) \,:=\, \sup_{\|f\|_{\mathbf{F}_d}\le 1} 
\abs{\int_{[0,1]^d}f(x)\dint x - Q_n(f)}.
\]
Furthermore, we define the spaces $\Ao$ as the collection of all $d$-variate functions from $\Aspt$ 
with support in the unit cube, cf.~\eqref{Ao} below. 
It seems that the cubature formula $Q_n$ of Frolov itself is suitable only for 
functions from $\Ao$, see Remark~\ref{rem:Frolov-Ao} below. 
However, in a forthcoming paper, see~\cite{NUU16}, we will present two general
modifications of $Q_n$ that define cubature formulas with non-equal weights providing the same order of
convergence also for periodic or non-periodic functions, respectively. 
Note that {\bf all} upper bounds on $\mbox{Int}_n(\Ao)$ that will be proven 
in this article are constructive. 
In particular, we will prove all upper bounds via 
$\mbox{Int}_n(\mathbf{F}_d)\lesssim e(Q_n, \mathbf{F}_d)$ for a specific 
cubature formula $Q_n$.
In fact, these cubature formulas will be $Q_n$ from~\eqref{eq:frolov}, for functions 
from~$\Ao$.

The main results of this article are the following Theorems~\ref{thm:error}~\&~\ref{thm:error2}.

\goodbreak

\begin{thm}\label{thm:error} Let $1\leq p,\theta \leq \infty$ (with $p<\infty$ in the $\T$-case). \\
If $s>1/p$, then 
\begin{equation}\label{thm:error-B}
   \mbox{Int}_n\bigl(\Bo\bigr) \,\asymp\, n^{-s}(\log n)^{(d-1)(1-1/\theta)}\quad,\quad n\geq 2\,.
\vspace{3mm}
\end{equation}
If $s>\max\{1/p,1/\theta\}$, then
\begin{equation}\label{thm:error-F}
    \mbox{Int}_n\bigl(\Fo\bigr) \,\asymp\, n^{-s}(\log n)^{(d-1)(1-1/\theta)}\quad,\quad n\geq 2\,.
\end{equation}
\end{thm}

The upper bound in relation~\eqref{thm:error-B} has been already proved by 
Dubinin~\cite{Du2}, see also the list of references below. However, our method is different (and simpler) and gives a
unified proof for Besov and Triebel-Lizorkin spaces. That is why we give also the relatively short proof for this case.

The lower bounds for $\theta = \infty$ are due to Bakhvalov \cite{Ba72}. The technique in the recent paper
Temlyakov \cite{Te14} also gives the lower bound for $\theta < \infty$.  Our proof of the
lower bounds in Section 7 is based on atomic decomposition and similar to the approach in \cite{DU14}. This approach
also gives the correct lower bounds in case $p<1$, see Theorem \ref{thm:lb}. 

In addition, we particularly pay attention to the case of small smoothness 
which is present for the spaces $\Fspt$ if $p>\theta$ and $1/p<s \leq 1/\theta$. 
It has been already observed by Temlyakov~\cite{Te93,Te94,Te03} that for 
the Fibonacci lattice rule on $\mathbb{T}^2$ the asymptotical worst-case error 
order with respect to the Sobolev class $\Wsp$
differs essentially in the $\log$-exponent from \eqref{thm:error-B} in the critical range of 
parameters $p > 2$ and $1/p<s\leq 1/2$. 
That is, Temlyakov~\cite{Te94} proved for this range that in dimension $d=2$
\[
e(\Phi_{n},\Wsp) \,\asymp\, \begin{cases} 
b_n^{-s}(\log b_n)^{1-s}, & \text{if}\quad 1/p<s<1/2,\\
b_n^{-s}(\log b_n)^{1-s}(\log\log b_n)^{1-s}, & \text{if}\quad s=1/2\,,
\end{cases}
\]
where $\Phi_n$ denotes the 
Fibonacci cubature rule with respect to the Fibonacci lattice with $b_n$ points, where $b_n$ denotes the $n$th
Fibonacci number. A matching lower bound for arbitrary cubature formulas 
is conjectured to hold, but is still not proven.
We were able to show a corresponding upper bound for $\Fo$ 
also in higher dimensions for the Frolov cubature rule. 
%by proving the following. 

\begin{thm}\label{thm:error2} Let $1\leq p<\infty$ and $1\leq \theta < p < \infty$.\\
{\em (i)} If $1/p<s<1/\theta$ then
$$
    \mbox{Int}_n\bigl(\Fo\bigr) \,\lesssim\, n^{-s}(\log n)^{(d-1)(1-s)}\quad,\quad n \geq 2\,.
$$
{\em (ii)} If $s = 1/\theta$ then 
$$
   \mbox{Int}_n\bigl(\Fo\bigr) \,\lesssim\, n^{-s}(\log n)^{(d-1)(1-s)}(\log\log n)^{1-s}\quad,\quad n \geq 3\,.
$$
\end{thm}
\noindent Concerning matching lower bounds there is so far 
only hope for $e(Q_n, \Fo)$. 
It is an interesting open problem to ask the same question for $\mbox{Int}_n(\Fspt)$ in 
the critical range of parameters. In any case, we strongly conjecture that the given convergence rates are sharp. This
would be another characteristic of a structural difference between the $\mathbf{B}$ and $\mathbf{F}$-spaces, recently
observed by Seeger and T. Ullrich \cite{SU15, SeUl15_2}, in the mentioned parameter domain.

Spaces of the above type have a long history in the former Soviet Union, 
see \cite{Am76, Nik75, ScTr87, Te93} and the references therein. 
The scale of spaces $\Bspt$ contains two important special cases of spaces 
with mixed smoothness: the H\"older-Zygmund spaces $(p=\theta=\infty)$ and 
the classical Nikol'skij spaces $(\theta=\infty)$. 
Note that Sobolev spaces $\Wsp$ with integrability $1<p<\infty$, $p\neq2$ and $s>0$ 
are not contained in the Besov scale. 
They represent special cases of Triebel-Lizorkin spaces $\Fspt$ 
if $\theta=2$. The upper bounds in Theorem~\ref{thm:error} were already 
obtained for 
\begin{itemize}
	\item $\Wsp(\mathbb{T}^2)$, $s>1/p$, see \cite[Thm.\ IV.2.1 \& IV.2.5]{Te93} and the references therein, 
	\item ${\mathbf B}^s_{p,\infty}(\mathbb{T}^2)$, $s>1/p$, see \cite[Thm.\ IV.2.6]{Te93}, and the
references therein, 
	\item $\Bspt(\mathbb{T}^2)$ for $s>1/p$, see \cite{TU14} and \cite{DU14}\,,
	\item $\Bspt([0,1]^2)$ for $1/p<s<1+1/p$, see \cite{TU14} and \cite{DU14},

	\item $\Bspt([0,1]^d)\cap\{f\colon f(x)=0 \text{ if }x_i=1 \text{ for some } i\}$ 
				for $s<1$, see \cite{Ma13}\,,
	\item $\Fspt([0,1]^d)\cap\{f\colon f(x)=0 \text{ if }x_i=1 \text{ for some } i\}$ 
				for $\max\{1/p,1/\theta\}<s<1$, see \cite{Ma13}\,,
	\item $\Wo([0,1]^d)$ for $s\in\N$, see~\cite{Fr76,Sk94},
	\item $\Wsp(\mathbb{T}^d)$ for $2\leq  p <\infty$ and $s\ge1$, see \cite[Thm.\ IV.4.4]{Te93}, 
  \item ${\mathbf B}^s_{p,\infty}(\mathbb{T}^d)$ if $s>1$ and $1<p\leq \infty$, see \cite[Thm.\ IV.4.6]{Te93}, and
	\item ${\mathbf B}^s_{p,\theta}(\mathbb{T}^d)$ for $1\le p,\theta\le
\infty$ and $s>1/p$, see~\cite{Du,Du2}; $1/p<s<2$, see \cite{HiMaOeUl14}.
\end{itemize}

The last four results were obtained by using Frolov(-type) cubature rules.
The other bounds are achieved for cubature formulas that use (digital) nets and 
the Fibonacci lattice rule, respectively, where the latter is 
restricted to $d=2$. For the Besov spaces with $d>2$ there are also some (not optimal) upper bounds 
by Triebel~\cite{Tr10} for $1/p<s<1+1/p$ using integration nodes from Smolyak grids or 
by Temlyakov \cite{Te86} using quasi-Monte Carlo lattice rules of Korobov type. 
Additionally, some results for Triebel-Lizorkin spaces can be easily obtained by 
embedding into $\Bspt$, see Lemma~\ref{emb} below. Concerning matching lower bounds for Sobolev spaces $\Wsp$ we
refer to \cite{Te90, Te94}, and for Besov spaces $\Bspt$ in case $s>1/p$
and $1\leq p,\theta \leq \infty$ to the recent paper \cite{DU14}. This will be complemented in Section \ref{sec:lower}
for the quasi-Banach case $\min\{p,\theta\}<1$. It shows that there will be no additional gain in the convergence
rate when passing to the situation $p<1$, an observation particularly relevant for the above mentioned kink functions.
In fact, we will present asymptotically optimal results for 
\begin{enumerate}
	\item[1)] $\mbox{Int}_n\bigl(\Ao\bigr)$ in the quasi-Banach cases $\min\{p,\theta\}<1$,
see~Section~\ref{subsec:quasi}, and
	\item[2)]$\mbox{Int}_n\bigl(\mathring{\A}^{1/p}_{p,\theta}\bigr)$ in the limiting case ($s=1/p$),
see~Section~\ref{subsec:limit}. 
\end{enumerate}

\noindent
Note that the results in Section~\ref{subsec:limit} require additional assumptions on 
$\theta$ to assure continuity of the functions. For the precise statement consider the 
mentioned section.

In addition, it is worth mentioning, that only recently Krieg and Novak~\cite{KN16} designed a Monte
Carlo algorithm that is based on the Frolov cubature rule and proved that this algorithm gains the order $1/2$ in the
main rate compared to the deterministic version studied in this paper. 

%Last but not least we discuss what happens in case $0<p<1$. We observe another main difference between Besov and Triebel-Lizorkin spaces. Namely, we prove the following. 
%
%\begin{thm} Let $0<p<1$, $0<\theta \leq \infty$ and $s>1/p$. Then we have
%$$
    %\mbox{Int}_N(\Bspt) \asymp N^{-s+1/p-1}(\log N)^{(d-1)(1-1/\theta)_+}\quad,\quad N\geq 2
%$$
%and
%$$
  %\mbox{Int}_N(\Fspt) \asymp N^{-s+1/p-1}\quad,\quad N \geq 1\,.
%$$
%\end{thm}
%\noindent In particular, if $p<1$ we do not observe the the ``small smoothness effect'' anymore. 
%
The paper is organized as follows. In Section \ref{sec:frolov} we introduce the Frolov cubature 
rule in detail. In addition, we give a direct construction of the Frolov lattice matrix without 
matrix inversion. Afterwards, in Section \ref{Prel}, we collect several tools from harmonic 
analysis. 
In particular, Poisson's summation formula and Calderon's reproducing formula 
in connection with local mean characterizations of function spaces turn out to be 
crucial for the error analysis of the cubature formula. 
The function spaces of interest will be introduced in 
Section~\ref{sec:spaces}. 
In Sections~\ref{sec:results}~and~\ref{sec:quasi} we prove the upper bounds 
for Frolov's cubature rule in the spaces $\Ao$ for the Banach and the quasi-Banach situation, 
respectively.
The matching (except for the case of small smoothness) lower bounds will be proven in Section~\ref{sec:lower}.

{\bf Notation.} As usual $\N$ denotes the natural numbers, $\N_0=\N\cup\{0\}$, 
%$\N_{-1}=\N_0\cup\{-1\}$, 
$\Z$ denotes the integers, 
$\R$ the real numbers, 
and $\C$ the complex numbers. By $\mathbb{T}$ 
we denote the torus represented by the interval $[0,1]$, where the end points are identified.
The letter $d$ is always reserved for the underlying dimension in $\R^d, \Z^d, \tor$ etc. We denote
with $\langle x,y\rangle$ or $x y$ the usual Euclidean inner product in $\R^d$ and $\C^d$. 
For $a\in \R$ we denote $a_+ := \max\{a,0\}$ and $a=\lfloor a\rfloor+\{a\}$ with 
$\lfloor a\rfloor\in\Z$ and $0\le\{a\}<1$. 
For $0<p\leq \infty$ and $x\in \R^d$ we denote $|x|_p = (\sum_{i=1}^d |x_i|^p)^{1/p}$ with the
usual modification in the case $p=\infty$. 
%We further denote $x_+ := ((x_1)_+,...,(x_d)_+)$ and $|x|_+ := |x_+|_1$. 
By $(x_1,\ldots,x_d)>0$ we mean that each coordinate is positive and $\lfloor x\rfloor$ 
(resp.~$\{x\}$) are meant component-wise. 
By $\mathbb{T}$ we denote the torus represented by the interval $[0,1]$.
If $X$ and $Y$ are two (quasi-)normed spaces, the (quasi-)norm
of an element $x$ in $X$ will be denoted by $\|x\|_X$. 
The symbol $X \hookrightarrow Y$ indicates that the identity operator is continuous. 
For two sequences of real numbers $a_n$ and $b_n$ we will write 
$a_n \lesssim b_n$ if there exists a constant $c>0$ such that 
$a_n \leq c\,b_n$ for all $n$. We will write $a_n \asymp b_n$ if 
$a_n \lesssim b_n$ and $b_n \lesssim a_n$.

%%%%%%%%%%%%%%%%%%%%%%%%%%%%%%%%%%%%%%%%%%%%%%%%%%%%%%%%%%%%
\goodbreak

\section{Frolov's construction}\label{sec:frolov}

In this section we introduce the matrix $T$ that is used in the construction 
of the lattice $\X_n$ for the cubature rule $Q_n$ from \eqref{eq:frolov}.
We follow to a large extend the original work of Frolov~\cite{Fr76}, 
see also~\cite{Te93,MU14}. Let
\begin{equation} \label{eq:poly}
P_d(t) \,:=\, \prod_{j=1}^d \bigl(t-2j+1\bigr) -1 
\,=\, \prod_{j=1}^d (t-\xi_j), \qquad t\in\R, 
\end{equation}
for suitable $\xi_1,\dots,\xi_d\in\R$, i.e.~the $\xi_j$ are the roots of $P_d$.
It is easy to see that the $\xi_j$ are indeed real numbers and all different.
Now define the matrix
\begin{equation}\label{eq:T}
\wt T \,=\, \bigl(\xi_i^{j-1}\bigr)_{i,j=1}^d
\end{equation}
and let $T_n=\bigl(n\det(\wt T)\bigr)^{-1/d}\,\wt T$ such that $\det(T_n)=1/n$. 
The lattice for the cubature formula $Q_n$ is given by $\X_n=T_n(\Z^d)$.

Later we will see that the error of the cubature formula $Q_n$ applied 
to some function $f$ is given by a sum over function evaluation of the 
Fourier transform $\F f$ at the points $B_n(\Z^d)$, where $B_n$ is 
given by 
\begin{equation}\label{eq:B}
B_n \,=\, \bigl(T_n^{-1}\big)^\top.
\end{equation}
The lattice $B_n(\Z^d)$ is usually called the \emph{dual lattice} of $\X_n$.
The following result contains the most important property of $B_n$.

\begin{lem}\label{lem:B}
Let $B_n$ be as above. Then, for each $z\in B_n(\Z^d)\setminus\{0\}$, we have 
\begin{equation}\label{hyp}
\prod_{j=1}^d|z_j| \,\gtrsim\, n.
\end{equation}
\end{lem}
\bproof
%Clearly, it is enough to prove the result for $n=1$.
First note that it is known that $|\prod_{j=1}^d(\wt Tm)_j|\ge1$ for 
each $m\in\Z^d\setminus\{0\}$ with $\wt T$ from \eqref{eq:T}, see~\cite[Lemma~IV.4.3]{Te93} 
or \cite[Lemma~8]{MU14}. 
Since $\wt T$ is a Vandermonde matrix we can write $\wt T^{-1}$ as the 
product $H \wt T^\top D^{-1}$, where $H$ is a Hankel matrix with integer 
entries and $\det(H)=-1$, and $D$ is a diagonal matrix with $\det(D)=\det(\wt T)^{2}$, 
see~\cite[Section~4]{LR04}.
Hence, 
\[\begin{split}
\inf_{m\in\Z^d\setminus\{0\}}\, \prod_{j=1}^d |(B_n m)_j|
\,&=\, n \det(\wt T)\,\inf_{m\in\Z^d\setminus\{0\}}\, \prod_{j=1}^d |(D^{-1}\wt T H^\top m)_j| \\
\,&=\, \frac{n}{\det(\wt T)}\,\inf_{m\in\Z^d\setminus\{0\}}\, \prod_{j=1}^d |(\wt T H^\top m)_j| 
\,\ge\, \frac{n}{\det(\wt T)}\,\inf_{m\in\Z^d\setminus\{0\}}\, \prod_{j=1}^d |(\wt T m)_j| \\
\,&\ge\, \frac{n}{\det(\wt T)},
\end{split}\]
where we have used that $H^\top(\Z^d)\subset\Z^d$.
\eproof

Note that lattices which satisfy the conclusion of Lemma~\ref{lem:B} 
satisfy $|\X_n\cap[0,1)^d|= n + O\bigl(\ln^{d-1}(n)\bigr)$, 
see \cite{Sk94}.

\begin{rem}\label{rem:admissible}
The original construction of the lattice $\X_n$ of Frolov~\cite{Fr76}, 
see also~\cite{Te93}, uses the generator $n^{-1/d} B_1$ instead of $T_n$. 
As it turns out we can use \eqref{eq:T} directly and do
not have to invert the matrix $\wt T$ for the construction of the nodes.
However, this ``trick'' only works if the generator $\wt T$ is 
(a multiple of) a Vandermonde matrix.
\end{rem}
The crucial property of the matrix $B_n$, which we will need in the sequel, is reflected best
by the numbers 
\begin{equation}\label{eq:Z}
Z_n(m) \,=\, \abs{(B_n(\Z^d)\setminus\{0\})\,\cap\, I_m},
\end{equation}
with 
\begin{equation}\label{eq:dyadic}
I_m:=\{x\in\R^d\colon C_1 \lfloor2^{m_j-1}\rfloor \le |x_j|< C_2\, 2^{m_j} 
\text{ for } j=1,\dots,d\}\,,\quad m\in\N_0^d,
\end{equation}
with constants $0<C_1\le C_2<\infty$ independent of $m$. The numbers $Z_n(m)$ denotes the number of points from
$B_n(\Z^d)$ excluding $0$ in 
(the union of at most $2^d$) axis-parallel boxes with volume of approximately $2^{|m|_1}$. The following estimates for
$Z_n(m)$ are a direct consequence of \eqref{hyp} and the lattice structure of $B_n(\Z^d)$. For a proof we refer to
\cite[IV.4]{Te93} or \cite{MU14}.

\begin{lem} \label{lem:Z}
There exists an absolute constant $c<\infty$ such that with 
\begin{equation}\label{eq:rn}
r_n:=\log_2(n)-c
\end{equation}
the numbers $Z_n(m)$ from~\eqref{eq:Z}, $n\in\N$, $m\in\N_0^d$, satisfy
\begin{enumerate}[(i)]
	\item $Z_n(m)\,=\,0$,\qquad\qquad if $|m|_1\le r_n$, and
	\item $Z_n(m)\,\lesssim\, 2^{|m|_1}/n$,\quad otherwise.
\end{enumerate}
\end{lem}

\begin{rem}
The choice of the polynomial \eqref{eq:poly} is quite flexible. 
One could replace $P_d$ by every irreducible (over $\mathbb{Q}$) polynomial 
with integer coefficients that has $d$ different real roots.
Another example, for $d=2^\ell$, is the Chebychev-type polynomial 
$P_{2^\ell}(t)=2 \cos(2^\ell\arccos(t/2))$, see \cite[p,~242]{Te93}. 
The roots of these polynomials are given by 
$\xi_i=2\cos(\pi(2i-1)/2^{n+1})$, $i=1,\dots,d$.
\end{rem}

\begin{rem}
It is worth noting that the results of this paper do not rely on the specific 
construction. 
In fact, every lattice for which the dual lattice satisfies 
the bound of Lemma~\ref{lem:B} would lead to the same results. 
See \cite{Sk94} for a detailed study of such lattices.
\end{rem}

\goodbreak
%%%%%%%%%%%%%%%%%%%%%%%%%%%%%%%%%%%%%%%%%%%%%%%%%%%%%%%%%%%%%%%%

\section{Tools from Fourier analysis}
\label{Prel}

\subsection{Preliminaries}
\noindent
Let $L_p=L_p(\R^d)$, $0 < p\le\infty$, be the space of all measurable functions 
$f:\R^d\to\C$ such that 
\[
\|f\|_p := \Big(\int_{\R^d} |f(x)|^p \dint x \Big)^{1/p} < \infty
\]
with the usual modification if $p=\infty$. In addition, we denote by $C(\R^d)$
the space of all bounded and continuous complex-valued functions on $\R^d$.

We will also need $L_p$-spaces on compact domains $\Omega\subset\R^d$ 
instead of $\R^d$. 
We write $\|f\|_{L_p(\Omega)}$ for the corresponding (restricted) $L_p$-norm. 
For $f\in L_1(\R^d)$ we define the Fourier transform
\[
\F f(\xi) 
\,=\, \int_{\R^d} f(y) e^{-2\pi i \langle \xi, y\rangle} \dint y, \qquad \xi\in\R^d,  
\]
and the corresponding inverse Fourier transform $\F^{-1}f(\xi)=\F f(-\xi)$.
Additionally, we define the spaces of 
continuous functions $C(\R^d)$, infinitely differentiable functions $C^\infty(\R^d)$ 
and infinitely differentiable functions with compact support $C^\infty_0(\R^d)$ 
as well as 
the \emph{Schwartz space} $\S=\S(\R^d)$ of 
all rapidly decaying infinitely differentiable functions on $\R^d$, i.e.,  
\[
\S := \bigl\{\varphi\in C^{\infty}(\R^d)\colon \|\varphi\|_{k,\ell}<\infty 
\;\text{ for all } k,\ell\in\N\bigr\}\,,
\]
where 
$$
    \|\varphi\|_{k,\ell}:=\Big\|(1+|\cdot|)^k\sum\limits_{m=0}^{\ell}|\varphi^{(m)}(\cdot)|\Big\|_{\infty}\,.
$$

The space $\mathcal{S}'(\R^d)$, the topological dual of $\mathcal{S}(\R^d)$, is also referred to as the set of tempered
distributions on $\R^d$. Indeed, a linear mapping $f:\mathcal{S}(\R^d) \to \C$ belongs
to $\mathcal{S}'(\R^d)$ if and only if there exist numbers $k,\ell \in \N$
and a constant $c = c_f$ such that
\begin{equation}\label{eq100}
    |f(\varphi)| \leq c_f\|\varphi\|_{k,\ell}
\end{equation}
for all $\varphi\in \mathcal{S}(\R^d)$. The space $\mathcal{S}'(\R^d)$ is
equipped with the weak$^{\ast}$-topology.
The convolution $\varphi\ast \psi$ of two
%integrable (square integrable) 
square-integrable
functions $\varphi, \psi$ is defined via the integral
\begin{equation*}\label{conv}
    (\varphi \ast \psi)(x) = \int_{\R^d} \varphi(x-y)\psi(y)\,dy\,.
\end{equation*}
If $\varphi,\psi \in \mathcal{S}(\R^d)$ then $\varphi \ast \psi$ still belongs to
$\mathcal{S}(\R^d)$. 
In fact, the convolution operator can be extended to $\mathcal{S}(\R^d)\times L_1$, 
in which case
we have $\varphi \ast \psi\in\S(\R^d)$, and to 
$\mathcal{S}(\R^d)\times \mathcal{S}'(\R^d)$ via
$(\varphi\ast f)(x) = f(\varphi(x-\cdot))$. It makes sense point-wise and is 
a $C^{\infty}$-function in $\R^d$.
As usual, the Fourier transform can be extended to $\mathcal{S}'(\R^d)$
by $(\cf f)(\varphi) := f(\cf \varphi)$, where
$\,f\in \mathcal{S}'(\R^d)$ and $\varphi \in \mathcal{S}(\R^d)$. 
The mapping $\cf:\S'(\R^d) \to \S'(\R^d)$ is a bijection.

\subsection{Periodization and Poisson's summation formula}

The analysis of the error of cubature formulas that use nodes from a lattice 
is naturally related to 
an application of Poisson's summation formula and variations thereof, 
see~\eqref{eq:error}.
A more detailed treatment and a proof of the following theorem 
can be found; e.g., in~\cite[Thm.~VII.2.4 \& Cor.~VII.2.6]{SW71}.

\begin{thm}\label{thm:periodization}
Let $f\in L_1(\R)$. Then its periodization $\sum_{\ell\in\Z^d} f(\ell+\cdot)$ converges 
in the norm of $L_1([0,1]^d)$. The resulting (1-periodic) function in 
$L_1([0,1]^d)$ has the formal Fourier expansion
\[
\sum_{k\in\Z^d} \F f(k)\, e^{2\pi i k x}.
\]
Moreover, if $f\in\S(\R^d)$, then both sums converge absolutely, and 
hence point-wise.
\end{thm}

Note that absolute convergence in the above sums actually holds under 
weaker assumptions on $f$. However, this version is enough for our purposes.

In what follows, cf.~\eqref{eq:error}, we will need a point-wise version of Poisson's 
summation formula, like Theorem~\ref{thm:periodization}, in cases where we cannot 
guarantee that the sums converge absolutely. In such cases we have to specify 
in which sense (ordering) we understand convergence. This is provided by the corollary below.

\begin{cor}\label{cor:poisson}
Let $f\in L_1(\R^d)$ be continuous with compact support, 
$T\in\R^{d\times d}$ be an invertible matrix, and $B=(T^{-1})^\top$.
Furthermore, let $\phi_0\in C_0^\infty(\R)$ with $\phi_0(0)=1$ and define 
$\phi_j(t):=\phi_0(2^{-j}t)-\phi_0(2^{-j+1}t)$, $j\in\N$, $t\in\R$, as well as 
the (tensorized) functions $\phi_m(x):=\phi_{m_1}(x_1)\cdot\ldots\cdot\phi_{m_d}(x_d)$,
$m\in\N_0^d$, $x\in\R^d$. Then
\[
\det(T)\sum_{\ell\in \Z^d} f(T\ell) 
\,=\, \lim_{N\to\infty}\;\sum_{m\colon|m|_\infty\le N}\; \sum_{k\in\Z^d} 
\phi_m(Bk)\, \F f(B k).
\]
In particular, the limit on the right hand side exists.
\end{cor}

\bproof
The proof is based on \cite[Theorem~VII.2.11]{SW71}. We put $\Phi_B(\cdot):=\varphi_0(B\cdot)$ 
and note, that $\Phi_B$ satisfies the assumptions \cite[(2.10)]{SW71}, in particular 
$\Phi_B(0) = 1$. Moreover, $f$ is continuous and has compact support, which implies that 
the periodization $\det(T)\sum_{\ell\in \Z^d} f(T(\ell+x))$ belongs to $C(\tor)$. 
Applying \cite[Thm.~VII.2.11]{SW71} together with Theorem~\ref{thm:periodization} above
we obtain 
\[
\det(T)\sum_{\ell\in \Z^d} f(T(\ell+x)) 
\,=\, \lim_{N\to\infty} \sum_{k\in\Z^d} \Phi_B(2^{-N}k)\, \F f(B k)\, e^{2\pi i k x}
\]
in $C(\tor)$. Setting $x=0$ gives 
\begin{equation}\label{SW}
   \det(T)\sum_{\ell\in \Z^d} f(T\ell) = \lim_{N\to\infty} \sum_{k\in\Z^d} \Phi_B(2^{-N}k)\, \F f(B k).
\end{equation}
By construction we note that 
$$
    \Phi_B(2^{-N}k) = \phi_0(2^{-N}(Bk)_1)\cdot\ldots\cdot\phi_0(2^{-N}(Bk)_d) = \sum_{m\colon |m|_\infty\le N} \phi_m (Bk)\,.
$$
Plugging this into \eqref{SW} and interchanging the order of summation yields the desired result. \eproof

\medskip

To prove our main results we will also have to bound the norm of certain series of functions. 
In fact, we treat two different kinds of functions: 
The first bound in Lemma~\ref{lem:norm1} requires that the functions itself have compact support, 
while the second one in Lemma~\ref{lem:norm2} is for functions with compactly supported Fourier transform.

\begin{lem}\label{lem:norm1}
Let $B\in\R^{d\times d}$ be an invertible matrix and 
$\Omega\in\R^d$ be a bounded set. 
Furthermore, let $\{f_m\}_{m\in\N_0^d}\subset \S(\R^d)$ be functions with 
$\supp(f_m)\subset\Omega$ for all $m\in\N_0^d$
and define
\[
M_{B,\Omega} \,:=\, \abs{\{\ell\in\Z^d\colon (\ell+[0,1)^d)\cap B^\top(\Omega)\neq\varnothing\}}.
\] 
Then, for $1\le \theta, p\le\infty$, we have
\[
\Big\|\Big(\sum_{m\in\N_0^d}\Big|\sum_{\ell\in\Z^d}\F f_m(B \ell)\, \e^{2\pi i \ell \cdot}
	\Big|^\theta\Big)^{1/\theta} \Big\|_{L_{p}([0,1]^d)}
\,\le\, \Big(\frac{M_{B,\Omega}}{\det(B)}\Big)^{1-1/p}\, 
	\Big\|\Big(\sum_{m\in\N_0^d} |f_m|^\theta\Big)^{1/\theta} \Big\|_{p}.
\]
In particular, 
\[
\Big\|\sum_{\ell\in\Z^d}\F f(B \ell)\, \e^{2\pi i \ell \cdot}\Big\|_{L_{p}([0,1]^d)}
\,\le\, \Big(\frac{M_{B,\Omega}}{\det(B)}\Big)^{1-1/p}\, 
	\|f\|_{p}.
\]
for $f\in \S(\R^d)$ with $\supp(f)\subset\Omega$.
\end{lem}

\begin{proof}
Let $p'$ be given by $1/p'=1-1/p$. Let $h_m(x)=f_m(Tx)$ with $T=(B^{-1})^\top$
%and $M_\Omega:=|\{m\in\Z^d\colon (m+[0,1)^d)\cap B^\top(\Omega)\}|$. 
and note that $\supp(h_m)\subset B^\top(\Omega)$. Clearly, $\F f_m(B\ell)=\det(T)\F h_m(\ell)$ and, hence, by Theorem~\ref{thm:periodization} 
and H\"older's inequality 
we obtain
\[\begin{split}
&\Big\|\Big(\sum_{m\in\N_0^d}\Big|\sum_{\ell\in\Z^d}\F f_m(B \ell)\, \e^{2\pi i \ell \cdot}
	\Big|^\theta\Big)^{1/\theta} \Big\|_{L_{p}([0,1]^d)}
\,=\, \det(T)\,\Big\|\Big(\sum_{m\in\N_0^d}	\Big|\sum_{\ell\in\Z^d}h_m(\ell+\cdot) 
	\Big|^\theta\Big)^{1/\theta} \Big\|_{L_{p}([0,1]^d)} \\
&\qquad\qquad\qquad\le\, \det(T)\,\Big\|\sum_{\ell\in\Z^d}\Big(\sum_{m\in\N_0^d}	\Big|h_m(\ell+\cdot) 
	\Big|^\theta\Big)^{1/\theta} \Big\|_{L_{p}([0,1]^d)} \\
&\qquad\qquad\qquad\le\, \det(T)\,\Big\|M_{B,\Omega}^{1/p'} 
		\Big(\sum_{\ell\in\Z^d}\Big(\sum_{m\in\N_0^d}	\Big|h_m(\ell+\cdot) 
	\Big|^\theta\Big)^{p/\theta}\Big)^{1/p} \Big\|_{L_{p}([0,1]^d)}\,.
\end{split}
\]
Performing the integration and interchanging sum and integral yields

\[\begin{split}
&\Big\|\Big(\sum_{m\in\N_0^d}\Big|\sum_{\ell\in\Z^d}\F f_m(B \ell)\, \e^{2\pi i \ell \cdot}
	\Big|^\theta\Big)^{1/\theta} \Big\|_{L_{p}([0,1]^d)}\\
&\qquad\qquad\qquad\le\, \det(T)\,M_{B,\Omega}^{1/p'} 
	\Big(\int_{[0,1]^d}\sum_{\ell\in\Z^d}\Big(\sum_{m\in\N_0^d} 
	\Big|h_m(\ell+x)\Big|^\theta\Big)^{p/\theta} \dint x\Big)^{1/p} \\
&\qquad\qquad\qquad=\, \det(T)\,M_{B,\Omega}^{1/p'} 
	\Big(\sum_{\ell\in\Z^d}\int_{\ell+[0,1]^d}\Big(\sum_{m\in\N_0^d} 
	\Big|h_m(x)\Big|^\theta\Big)^{p/\theta} \dint x\Big)^{1/p} \\
&\qquad\qquad\qquad=\, \det(T)\,M_{B,\Omega}^{1-1/p} 
	\Big(\int_{\R^d}\Big(\sum_{m\in\N_0^d} 
	\Big|h_m(x)\Big|^\theta\Big)^{p/\theta} \dint x\Big)^{1/p} \\
&\qquad\qquad\qquad=\, 
%\Big(\det(T)\,M_{B,\Omega}\Big)^{1-1/p} 
	%\Big(\int_{\R^d}\Big(\sum_{m\in\N_0^d} 
	%\Big|f_m(x)\Big|^\theta\Big)^{p/\theta} \dint x\Big)^{1/p} 
%\,=\, 
\Big(\det(T)\,M_{B,\Omega}\Big)^{1-1/p} 
	\Big\|\Big(\sum_{m\in\N_0^d} |f_m|^\theta\Big)^{1/\theta} \Big\|_{p}.
\end{split}\]
The second statement follows if we set $f_0=f$ and $f_m=0$, $m\neq0$.\\
\end{proof}

\begin{rem}\label{rem:volume}
Note that $M_{B,\Omega}$ is the number of unit cubes in the standard tessellation 
of $\R^d$ that are necessary to 
cover the set $B^\top(\Omega)$, while $\det(B)$ equals the volume of 
$B^\top([0,1]^d)$. This shows that with a matrix of the form $B_n=n^{1/d} B$, 
$n\ge1$, cf.~\eqref{eq:B}, we obtain
\[
\lim_{n\to\infty}\, \frac{M_{B_n,\Omega}}{\det(B_n)} \,=\, {\rm vol}_d(\Omega)
\]
for every Jordan measurable set $\Omega$.
\end{rem}

\medskip

\begin{lem}\label{lem:norm2}
Let $B\in\R^{d\times d}$ be an invertible matrix and 
$\Omega\in\R^d$ be a bounded set. 
Furthermore, let $g\in\S(\R^d)$ with 
$\supp(\F g)\subset\Omega$.
Then, for $1\le p\le\infty$, we have
\[
\Big\|\sum_{k\in\Z^d} \F g(B k)\, \e^{2\pi i k \cdot}\Big\|_{L_{p}([0,1]^d)}
\,\le\, \abs{B(\Z^d)\cap\Omega}^{1-1/p}\, \|g\|_1.
\]
\end{lem}

\begin{proof}
Clearly, the proof of Lemma~\ref{lem:norm1} for $p=1$ works also for functions 
$f_m$ without compact support and, hence, already proves the statement for $p=1$. 
(Set $f_0=g$ and $f_m=0$ for $m\neq0$.)
For $p=\infty$ we easily obtain the upper bound 
$|B(\Z^d)\cap\Omega|\cdot \|\F g\|_\infty\le|B(\Z^d)\cap\Omega|\cdot \|g\|_1$. 
Hence, the lemma follows by H\"older's inequality.\end{proof}

%%%%%%%%%%%%%%%%%%%%%%%%%%%%%%%%%%%%%%%%%%%%%%%%%%%%%%%%%%%%%%%%%%%%%%%%%%%

\subsection{A discrete version of Calderon's reproducing formula}

Our analysis heavily relies on a discrete version of Calderon's reproducing formula \cite[eq.~(3.1)]{Ca77}. A ``continuous'' and homogeneous 
version of the Lemma below has been proved in \cite{JaTa82}. This principle has been used by several authors \cite{BuPaTa96}, \cite{FrJa90}, \cite{Ry99}, \cite{Ry99a} to prove equivalent (local mean) characterizations for Besov-Triebel-Lizorkin spaces, see Section \ref{sec:spaces}.

%\begin{lem}\label{lem:decomp}
%Let $\Psi_0, \Psi_1\in\S(\R)$ be functions with
%\[
%|\F\Psi_0(\xi)|>0\quad \text{ for } |\xi|<\eps
%\]
%and 
%\[
%|\F\Psi_1(\xi)|>0\quad \text{ for } \frac\eps2<|\xi|<2\eps
%\]
%for some $\eps>0$.
%Then, for some $0<\lambda<\eta\leq 2\lambda$, there exist admissible kernels 
%$\Lambda_0, \Lambda_1\in\S(\R)$, cf.~Definition~\ref{def:kernels}, 
%such that
%\begin{enumerate}[(\it i)]
	%\item $\supp\,\F\Lambda_0 \subset \{t\in \R:|t| \leq \eta\}$ and $|\F \Lambda_0(t)|>0$ on $\{t\in \R:|t| < \tilde{\eta}\}$\,,
	%\item $\supp\,\F\Lambda_1 \subset \{t\in \R:\lambda \leq |t| \leq \eta\}$ and $|\F \Lambda_1(t)|>0$ on $\{t\in \R:\tilde{\lambda}<|t| < 2\tilde{\eta}\}$\,,
	%\item and for all $t\in\R$, 
		%\begin{equation}\label{eq:Lambda}
		%\sum_{j=0}^\infty \F\Lambda_j(t)\, \F\Psi_j(t) \,=\, 1, 
 		%\end{equation}
		%where $\Psi_j(x)=2^{j-1}\Psi_1(2^{j-1}x)$ and 
		%$\Lambda_j(x)=2^{j-1}\Lambda_1(2^{j-1}x)$ for $j\in\N$.
%\end{enumerate}
%\end{lem}

\begin{lem}\label{lem:decomp}
Let $\Psi_0, \Psi_1\in\S(\R)$ be functions with
\begin{equation}\label{tc1}
|\F\Psi_0(\xi)|>0\quad \text{ for } |\xi|<\eps
\end{equation}
and 
\begin{equation}\label{tc2}
|\F\Psi_1(\xi)|>0\quad \text{ for } \frac\eps2<|\xi|<2\eps
\end{equation}
for some $\eps>0$. Then there exist $\Lambda_0, \Lambda_1\in \S(\R)$ such that
\begin{enumerate}[(\it i)]
	\item $\supp\,\F\Lambda_0 \subset \{t\in \R:|t| \leq \eps\}$
	\item $\supp\,\F\Lambda_1 \subset \{t\in \R:\eps/2 \leq |t| \leq 2\eps\}$ and
	\item for all $\xi\in\R$, 
		\begin{equation}\label{eq:Lambda}
		\sum_{j=0}^\infty \F\Lambda_j(\xi)\, \F\Psi_j(\xi) \,=\, 1, 
 		\end{equation}
		\noindent where $\Psi_j(x)=2^{j-1}\Psi_1(2^{j-1}x)$ and 
		$\Lambda_j(x)=2^{j-1}\Lambda_1(2^{j-1}x)$ for $j\in\N$.
\end{enumerate}
\end{lem}

\bproof Following \cite[Thm.\ 1.20]{Vyb06} we use the special dyadic 
decomposition of unity with $\varphi(t) = 1$ if $|t|\leq 4/3$ and 
$\varphi(t) =0$ if $|t|>3/2$. 
Put $\Phi_0:=\F^{-1}\varphi$ and $\Phi_1:=2\Phi_0(2\cdot)-\Phi_0$, 
i.e.~$\F\Phi_1=\Phi_0(\cdot/2)-\Phi_0$. 
With $\Phi_j:=2^{j-1}\Phi_1(2^{j-1}\cdot)$ for $j\geq 1$ we define $\Lambda_0, \Lambda_1$ through 
$$
    \F\Lambda_j(t) := \frac{\F \Phi_j(2t/\eps)}{\F \Psi_j(t)}\quad,\quad t\in \R\,.
$$    
\eproof
\noindent We define the $d$-fold tensorized functions 
$$
\Lambda_m(x) \,:=\, \prod_{i=1}^d \Lambda_{m_i}(x_i)\quad\mbox{and}\quad \Psi_m(x) \,:=\, \prod_{i=1}^d \Psi_{m_i}(x_i)\quad,\quad x\in \R^d. 
$$
where $\Lambda_j, \Psi_j$, $j\in\N$, are defined in Lemma~\ref{lem:decomp}.
%With the functions from Lemma \ref{lem:decomp} w
We obtain from \eqref{eq:Lambda} the identity 
\begin{equation}\label{eq:calderon}
\sum_{m\in\N_0^d} \F\Lambda_m(\xi)\, \F\Psi_m(\xi) \,=\, 1\quad,\quad \xi \in \R^d\,.
\end{equation}
By the construction of the tensorized functions (and Lemma~\ref{lem:decomp}) 
we know that the support of $\F\Lambda_m$ is of the form \eqref{eq:dyadic} 
and we will write in the sequel
\[
I_m \,:=\, \supp\,\F\Lambda_m, \qquad m\in\N_0^d.
\]

%\subsection{The up-function}
%\label{up}
%
%We make use of a special infinitely many times differentiable compactly supported function.
%This (univariate) function is known as the \emph{up-function}, see 
%e.g. Rvachev~\cite{Rv90} or \cite[Section 6.1]{HNUW14}, and has many useful properties.
%
%Let $g(t)=\ind_{[-1/2,1/2]}(t)$, $t\in\R$, be the characteristic function of 
%the interval $[-1/2,1/2]$ and define $g_k(t)=2^k g(2^k t)$. 
%The up-function $\phi$ is now given by the infinite convolution
%\[
%\phi \,=\, g \ast g_1 \ast g_2 \ast\dots.
%\]
%Obviously, $\phi\in C_0^\infty(\R)$, $0\le\phi\le1$, $\|\phi\|_1=\|g\|_1=1$ and 
%$\supp(\phi)=[-1,1]$.
%It is known that $\F g(\xi)= \sinc(\xi):=\sin(\pi \xi)/(\pi \xi)$ and hence, 
%$\phi$ could be equivalently defined by 
%\[
%\F\phi(\xi)\,=\, \prod_{k=0}^\infty\, \sinc(2^{-k} \xi), \qquad \xi\in\R.
%\]
%Moreover, $\phi$ is a solution to the equation $\phi'(t) = 2\phi(2t+1)-2\phi(2t-1)$ 
%and satisfies therefore
%\[
%\phi(t) + \phi(t-1) \,=\, 1 \qquad \text{ for all } t\in[0,1],
%\]
%see~\cite[Paragraph 1]{Rv90}.
%To see this, observe that 
%$\int_0^1\phi(t)+\phi(t-1)\dint t = \int_{-1}^1\phi(t)\dint t =1$ and that the 
%differential equation (together with $\supp(\phi)=[-1,1]$) implies 
%$\phi'(t) + \phi'(t-1)=2\phi(2t+1)-2\phi(2t-3)=0$ for $t\in[0,1]$.
%

\section{Function spaces with dominating mixed smoothness}
\label{sec:spaces}

In this section we introduce the function spaces under consideration, namely, the Besov and Triebel-Lizorkin spaces of dominating mixed smoothness. 
Note that the Sobolev spaces of mixed smoothness appear as a special case of the 
Triebel-Lizorkin spaces. There are several equivalent characterizations of these spaces, 
see \cite{Vyb06}. For our purposes, the most suitable 
is the characterization by local means (see \cite[Theorem~1.23]{Vyb06} or 
\cite[Definition~2.5]{TU08}). 

%\subsection{Spaces on $\R^d$}
\medskip

We start with the definition of the spaces on $\R^d$.

\noindent 
Let $\Psi_0,\Psi_1\in C_0^\infty(\R)$ be such that 
\begin{enumerate}
%	\item[$(i)$] $\supp(\Psi_i)\subset[-1,1]$ for $i=0,1$,
	\item[$(i)$] $|\F\Psi_0(\xi)|>0$\qquad for $|\xi|<\eps$,
	\item[$(ii)$] $|\F\Psi_1(\xi)|>0$\qquad for $\frac\eps2<|\xi|<2\eps$ and 
	\item[$(iii)$] $D^\a\F\Psi_1(0)=0$\qquad for all $0\le\a\le L$
\end{enumerate}
for some $\eps>0$. A suitable $L$ will be chosen in Definitions~\ref{def:besov}~\&~\ref{def:TL}. 
As usual, we define 
\[
\Psi_j(x) \,=\, 2^{j-1}\, \Psi_1(2^{j-1} x), \qquad j\in\N, 
\]
and the ($d$-fold) tensorization
\begin{equation}\label{eq:Psi}
\Psi_m(x) \,=\, \prod_{i=1}^d \Psi_{m_i}(x_i),
\end{equation}
where $m=(m_1,\dots,m_d)\in\N_0^d$ and $x=(x_1,\dots,x_d)\in\R^d$. 

\begin{rem}\label{up} There exist compactly supported functions $\Psi_0, \Psi_1$ satisfying (i)-(iii) above. 
Consider $\Psi_0$ to be the {\em up-function}, 
see Rvachev~\cite{Rv90}. This function satisfies $\Psi_0\in C^{\infty}_0(\R)$ with $\supp(\Psi_0)=[-1,1]$
and $\F\Psi_0(\xi) = \prod_{k=0}^\infty\, \sinc(2^{-k}\xi)$, $\xi\in\R$, where $\sinc$ denotes the normalized sinus
cardinalis $\sinc(\xi) = \sin(\pi \xi)/(\pi \xi)$. If we define $\Psi_1\in C_0^\infty(\R)$ to be
\[
\Psi_1(x) \,:=\, \frac{d^L}{dx^L}\bigl(2\Psi_0(2\cdot)-\Psi_0(\cdot)\bigr)(x), 
\qquad x\in\R\,,
\]
it follows that
$
\F\Psi_1(\xi) \,=\, (2\pi i \xi)^{L}
\bigl(\F\Psi_0(\xi/2) - \F\Psi_0(\xi)\bigr).
$
It is easily checked that these functions satisfy the conditions (i),(ii),(iii) above. In particular, (i) and (ii) is
satisfied with $\eps = 1$.
Moreover, we have for all $m\in\N_0^d$ that the tensorized functions 
$\Psi_m$ satisfy $\supp(\Psi_m)\subset[-1, 1]^d$. 
We will work with this choice in the sequel.
\end{rem}

Let us continue with the definition of the function spaces $\Aspt = \Aspt(\R^d)$ with 
$\A\in\{\mathbf{B},\mathbf{F}\}$ defined on the entire $\R^d$.

\begin{defi}[Besov space] \label{def:besov}
Let $0 < p,\theta\le\infty$, $s\in\R$, and 
$\{\Psi_m\}_{m\in\N_0^d}$ be as above with $L+1>s$. % in \eqref{vanmom}. 
The \emph{Besov space of dominating mixed smoothness} 
$\Bspt=\Bspt(\R^d)$ is the set of all $f\in \S'(\R^d)$ 
such that 
\[
\|f\|_{\Bspt} \,:=\, 
\Big(\sum_{m\in\N_0^d} 2^{s |m|_1 \theta}\, 
	\|\Psi_m \ast f\|_p^\theta\Big)^{1/\theta} \,<\,\infty
\]
with the usual modification for $\theta=\infty$. 
\end{defi}

\begin{defi}[Triebel-Lizorkin space] \label{def:TL}
Let $0< p<\infty$, $0 < \theta\le\infty$, $s\in\R$, and 
$\{\Psi_m\}_{m\in\N_0^d}$ be as above with $L+1>s$. % in \eqref{vanmom}. 
The \emph{Triebel-Lizorkin space of dominating mixed smoothness} 
$\Fspt=\Fspt(\R^d)$ is the set of all $f\in \S'(\R^d)$ 
such that 
\[
\|f\|_{\Fspt} \,:=\, 
\Big\|\Big(\sum_{m\in\N_0^d} 2^{s |m|_1 \theta}\, 
	|\Psi_m \ast f(\cdot)|^\theta\Big)^{1/\theta}\Big\|_p \,<\,\infty
\]
with the usual modification for $\theta=\infty$. 
\end{defi}

\begin{rem} In the special case $\theta=2$ and $1<p<\infty$ we put $\Wsp := \Fsptw$ which denotes the 
\emph{Sobolev spaces of dominating mixed smoothness}. It is well-known (cf.~\cite[Chapt.\ 2]{ScTr87} or \cite{Vyb06}), 
that in case $s\in\N_0$ the spaces $\Wsp$ can be equivalently normed by 
\[
\|f\|_{\Wsp} \,\asymp\, \Big(\sum_{\substack{\a\in\N_0^d\\ |\a|_\infty\le s}} \|D^\a f\|_p^p\Big)^{1/p}\,.
\]
 
\end{rem}

\begin{rem} Different choices of $\Psi_0, \Psi_1$ lead to equivalent (quasi-)norms. 
In fact, it is not even necessary that $\Psi_0$ and $\Psi_1$ have compact support. 
However, for the proof of our results this specific choice is crucial.
\end{rem}

\begin{rem}
Note, that the spaces $\Bspt(\R^d)$ and $\Fspt(\R^d)$ are usually defined via a dyadic 
decomposition of unity on the Fourier side (which represents a special case of the above 
framework) like the one we used in the proof of Lemma \ref{lem:decomp}, see \cite{Vyb06}. 
Here we used a different characterization which was proven to be equivalent, 
see \cite{Vyb06} and \cite{TU08}. 
Let us comment on the recent and non-trivial history of those characterizations. 
\begin{itemize}
\item In 1992 Triebel proved those characterizations for the isotropic spaces 
$F^s_{p,q}(\R^d)$, see \cite[2.4.2, 2.5.1]{Tr92}. He obtained characterizations for $p,\theta \ge 1$.

\item Later, in 1999, Rychkov \cite{Ry99a} extended it to the quasi-Banach case. However, there is a gap in his proof (observed 2007 by Hansen), see \cite{Ha10} and Remark 4.4 in \cite{TU08}.

\item In 2005 Vyb\'iral modified Rychkov's method for the dominating mixed case. Vyb\'iral's proof also contains Rychkov's gap. However, Vyb\'iral was the first who did the local means for dominating mixed spaces.

\item The proof in \cite{TU08} differs from Vyb\'iral's proof and fixes the mentioned gap, see also Hansen \cite{Ha10}. The proof is based on a maximal function technique, see \cite[pp. 20]{TU08}, due to Str\"omberg, Torchinsky \cite[Chapt. V]{StTo} which has been already proposed by Rychkov \cite[Thm.\ 3.2]{Ry99}. 

\item The proof in \cite{TU08} is a bit more general, namely for function spaces on semi-infinite rectangular domains. 
\end{itemize}
\end{rem}

\noindent The next lemma collects some frequently used embedding properties of the spaces.

\begin{lem}\label{emb} Let $0<p<u\le\infty$ ($p,u<\infty$ in the $\T$-case),
$s,t\in \R$, $0<\theta,\eta\leq \infty$.\\
{\em (i)} For equal $p$ we have the chain of embeddings 
$$
    \B^s_{p,\min\{p,\theta\}}\hookrightarrow \Fspt \hookrightarrow
\B^s_{p,\max\{p,\theta\}}
$$
{\em (ii)} In addition, whenever $s-1/p = t-1/u$ the following ``diagonal
embeddings'' hold true 
\begin{equation*}
       \B^s_{p,\theta} \hookrightarrow \B^t_{u,\theta}\quad,\quad
\T^s_{p,\theta} \hookrightarrow \T^t_{u,\eta}
 \end{equation*}
{\em (iii)} as well as the embeddings 
%of Jawerth-Franke type
\begin{equation*}
       \T^s_{p,\theta} \hookrightarrow \B^t_{u,p}\quad,\quad
\B^s_{p,u} \hookrightarrow \T^t_{u,\eta}\,.
 \end{equation*}
\end{lem}

\bproof For a proof we refer to \cite[Chapt.\ 2]{ScTr87}, as well as \cite{D} and
\cite{E}.
\eproof

\begin{rem} The embeddings in Lemma \ref{emb}(iii) are highly non-trivial. Let us comment on the history
of these important embeddings. Although they are referred to as Jawerth-Franke embeddings several
different people contributed to this result, especially what concerns Besov-Lizorkin-Triebel spaces with dominating
mixed smoothtness. 

\begin{itemize}
 \item An analog of the first embedding in Lemma \ref{emb}(iii) for isotropic Besov-Triebel-lizorkin spaces
$F^s_{p,q}(\R^d)$ and $B^s_{p,q}(\R^d)$ has been obtained by Jawerth \cite{A}, the second one by Franke \cite{H}. The
proof was based on real interpolation. 
\item A new proof of both relations based on atomic decompositions has
been given recently by Vyb\'iral \cite{B} without using real interpolation techniques. 
\item Hansen and Vyb\'iral \cite{E} extended this technique to prove the respective embeddings for spaces
with dominating mixed smoothness $\Bspt$ and $\Fspt$. Note, that the technique used by Jawerth and Franke (using real
interpolation) does not apply for spaces of mixed smoothness. 
\item A special case of the embeddings in Lemma \ref{emb}(iii) with the appropriate Sobolev type spaces $\Wsp(\tor)$
(where $\theta = 2$) instead of $\Fspt$ is due to Temlyakov \cite{C,D}.
\item The second embedding in Lemma \ref{emb}(iii) in the special case when $\Fspt$ is replaced by $L_u$ in the
univariate univariate case goes back to Ul'yanov \cite{F} and Timan \cite{G}. 
\end{itemize}
\end{rem}

In the sequel we will always assume that $s> 1/p$ 
%(or $s=1/p$ with additional assumptions on $\theta$). 
This assures that the
functions in $\Bspt$, $\Fspt$ and $\Wsp$, respectively, are continuous, see
\cite[Chapt.\ 2]{ScTr87}. With the same reasoning we obtain that
$\B^{1/p}_{p,\theta} \hookrightarrow C(\R^d)$ for $\theta\le1$, 
and $\T^{1/p}_{p,\theta} \hookrightarrow C(\R^d)$ for $p\le1$.

\medskip
 
%\subsection{Spaces on the cube}

The class of functions we are interested in throughout this article are 
the subclasses 
of functions from $\Aspt(\R^d)$ which are supported in the unit cube $[0,1]^d$, 
%We define for $s>\max\{1/p-1,0\}$
i.e.~we consider the classes 
\begin{equation}\label{Ao}
\Ao \,:=\, \bigl\{ f\in \Aspt(\R^d)\colon \supp(f)\subset[0,1]^d\bigr\}
\end{equation}
for $\A\in\{\B,\T\}$.

%%%%%%%%%%%%%%%%%%%%%%%%%%%%%%%%%%%%%%%%%%%%%%%%%%%%%%%%%%%%%%%%

\section{Integration of functions from $\Ao$}
\label{sec:results}

Here, we study the cubature formula $Q_n$ from \eqref{eq:frolov}, see Section~\ref{sec:frolov}, 
for the spaces $\Ao$ defined in \eqref{Ao}. Our main results read as follows. 

\begin{thm}\label{thm:besov}
For each $1\le p,\theta\le\infty$
%, $0<\theta\le\infty$ 
and $s>1/p$, we have 
%\[
%e(Q_n, \Bo) \,\asymp\, n^{-s} (\log n)^{(d-1)(1-1/\theta)}.
%\]
\[
e(Q_n, \Bo) \,\asymp\, n^{-s} (\log n)^{(d-1)(1-1/\theta)}.
\]
\end{thm}

\bigskip
\goodbreak

\begin{thm}\label{thm:TL}
For each $1\le p < \infty$, $1\le\theta\le\infty$ 
and $s>1/p$ the following bounds hold true.
\vspace{2mm}
\begin{enumerate}[\it (i)]
	\item If $s>1/\theta$ then 
		\[
		e(Q_n, \Fo) \;\asymp\; n^{-s}\, (\log n)^{(d-1)(1-1/\theta)}\,.
		\]
	\item If $s<1/\theta$ then
		\[
		e(Q_n, \Fo) \;\lesssim\; n^{-s}\, (\log n)^{(d-1)(1-s)}\,.
		\]
	\item If $s=1/\theta$ then
		\[
		e(Q_n, \Fo) \;\lesssim\; n^{-s}\, (\log n)^{(d-1)(1-s)}\, 
		(\log\log n)^{1-s}\,.
		\]
\end{enumerate}
\end{thm}

\bigskip

Setting $\theta=2$ immediately implies the following important special case.

\begin{cor}\label{cor:sobolev}
For each $1< p<\infty$ and $s>1/p$ the following bounds hold true.
\vspace{2mm}
\begin{enumerate}[\it (i)]
	\item If $s>1/2$ then
		\[
		e(Q_n, \Wo) \;\asymp\; n^{-s}\, (\log n)^{(d-1)/2}\,.
		\]
	\item If $s<1/2$ then 
		\[
		e(Q_n, \Wo) \;\lesssim\; n^{-s}\, (\log n)^{(d-1)(1-s)}\,.
		\]
	\item If $s=1/2$ then 
		\[
		e(Q_n, \Wo) \;\lesssim\; n^{-s}\, (\log n)^{(d-1)/2}\, 
		\sqrt{\log\log n}\,.
		\]
\end{enumerate}
\end{cor}
\medskip

\begin{rem}\label{rem:Frolov-Ao}
The subclasses $\Ao$ of functions with support inside the unit cube $[0,1]^d$ represent the ``natural domain'' for the Frolov cubature rule. 
This comes from the fact that the proof heavily relies on the use of 
Poisson's summation formula (Corollary~\ref{cor:poisson}), which requires 
that we can replace the summation in \eqref{eq:frolov} over 
$\X_n\cap[0,1)^d$ by a summation over $\X_n$ without changing the value of the sum, 
see~\eqref{eq:Qn}. 
\end{rem}

\begin{rem} Let us give some historical comments on the special case in Corollary \ref{cor:sobolev}. This result
was partially known before and has a long history, starting with the work of Bakhvalov~\cite{Ba63}. 
He proved a version of Corollary~\ref{cor:sobolev}(i) for the Fibonacci cubature rule $\Phi_n$ 
instead of $Q_n$ for the periodic spaces $\Wsp(\mathbb{T}^2)$ in case $s\in \N$ and $d=2$. 
This was extended by Temlyakov to $1<p<\infty$. In fact, Temlyakov~\cite{Te91, Te94} proved 
a version of Corollary \ref{cor:sobolev}(i)--(iii) with $Q_n$ replaced by $\Phi_n$ for
periodic spaces $\Wsp(\mathbb{T}^2)$ in case $d=2$ and $s>1/p$. 
Corollary~\ref{cor:sobolev}(i) in case $d>2$,
$p=2$, and $s\in \N$ is due to Frolov \cite{Fr76} which trivially implies the correct order 
for $p>2$ due to embedding.
For $d>2$, $1<p<\infty$ and $s\in \N$, we refer to Skriganov~\cite{Sk94}. 
All the remaining cases for $d>2$
%, in particular (i) in case $1<p<2$ and $s>1/p$, 
are new.
In addition, the results for small smoothness, see (ii) and (iii),  have not been known before. 
Also the recent paper Hinrichs, Markhasin,
Oettershagen, T.~Ullrich~\cite{HiMaOeUl14} provides the correct order of 
$\mbox{Int}_n(\Wsp(\tor))$
in case $1<p<2$ and
$1/p<s<2$ with a quasi-Monte Carlo method based on order-$2$ digital nets. For more detailed historical comments see
the monograph/survey Temlyakov \cite{Te93, Te03}. 
\end{rem}

\subsection{The proofs} \label{subsec:proofs}

We begin this section with the derivation of an error formula that 
is the starting point for the proofs in the specific cases.
This explicit formula for $|I(f)-Q_n(f)|$ follows immediately from 
the point-wise version of Poisson's summation formula for general lattices, 
see Corollary~\ref{cor:poisson}.

In this section we consider $f\in\Ao$, i.e.~functions with support in the unit cube. 
Hence, we can rewrite our cubature rule $Q_n$ from \eqref{eq:frolov} as
\begin{equation}\label{eq:Qn}
Q_n(f) \,=\,  \frac{1}{n}\sum_{x\in \X_n\cap[0,1)^d} f(x) 
\,=\,  \det(T_n)\sum_{x\in \X_n} f(x).
\end{equation}

In the sequel we will use the specific kernel from Remark \ref{up} and its tensorized version $\Psi_m$, $m\in
\N_0^d$\,. The corresponding functions  $\Lambda_m$, $m\in\N_0^d$, are given by Calderon's reproducing
formula \eqref{eq:Lambda}, \eqref{eq:calderon}. The construction of these functions, cf.~the proof of
Lemma~\ref{lem:decomp}, 
assures that the functions $\phi_m:=\F[\Lambda_m\ast\Psi_m]$ satisfy 
the assumptions of Corollary~\ref{cor:poisson} and hence,
\[
Q_n(f)
\,=\, \sum_{m\in\N_0^d}\; \sum_{k\in\Z^d} 
\F[\Lambda_m\ast\Psi_m](B_n k)\, \F f(B_n k).
\]
Note that the inner sum is finite and that, actually, the outer sum is defined 
as a certain limit, see Corollary~\ref{cor:poisson}. 
However, we will see that this sum converges absolutely in all cases under consideration. 

Note that $\ls\e^{2\pi i k \cdot}, \e^{2\pi i \ell \cdot} \rs=1$, if $k=\ell$, 
and 0 otherwise, where $\ls\cdot,\cdot\rs$ is usual inner product in $L_2([0,1]^d)$. 
Using this together with $I(f)=\F f(0)$ we obtain
\begin{equation}\label{eq:error}
\begin{split}
|I(f)-Q_n(f)| 
\,&=\, \abs{\sum_{m\in\N_0^d}\,\sum_{k\neq0} \F\Lambda_m(B_n k)\, \F\Psi_m(B_n k)\, \F f(B_n k)}\\
&=\,\abs{\sum_{m\in\N_0^d}\,\sum_{k\neq0}\,\sum_{\ell\neq0} \F\Lambda_m(B_n k)\, 
	\F[\Psi_m\ast f](B_n \ell)\; \l\e^{2\pi i k \cdot}, \e^{2\pi i \ell \cdot} \r}\\
&=\,\abs{\sum_{m\in\N_0^d} \l\sum_{k\neq0} \F\Lambda_m(B_n k)\, \e^{2\pi i k \cdot},
	\sum_{\ell\neq0}\F[\Psi_m\ast f](B_n \ell)\, \e^{2\pi i \ell \cdot} \r}.\\
\end{split}
\end{equation}

Now we have to proceed differently depending on the space under consideration. 
In fact, we will perform in either case H\"older's inequality twice, 
but in a different order.

\subsection*{The result for $\Bo$}

We now prove Theorem~\ref{thm:besov}.

Using~\eqref{eq:error} we obtain by H\"older's inequality
\begin{equation*}\label{eq:error-B}
\begin{split}
|I(f)-Q_n(f)| \,&=\, \abs{\sum_{m\in\N_0^d} \l\sum_{k\neq0} \F\Lambda_m(B_n k)\, \e^{2\pi i k \cdot},
	\sum_{\ell\neq0}\F[\Psi_m\ast f](B_n \ell)\, \e^{2\pi i \ell \cdot} \r}\\
&\le\, \sum_{m\in\N_0^d}\, 
	\Big\|\sum_{k\neq0} \F\Lambda_m(B_n k)\, \e^{2\pi i k \cdot}\Big\|_{L_{p'}([0,1]^d)}\, 
	\Big\|\sum_{\ell\neq0}\F[\Psi_m\ast f](B_n \ell)\, \e^{2\pi i \ell \cdot} \Big\|_{L_{p}([0,1]^d)}\\
\end{split}
\end{equation*}
with $1/p+1/p'=1$.

\medskip

Using Lemma~\ref{lem:norm2} for the first and Lemma~\ref{lem:norm1} 
for the second factor, we obtain
\[
|I(f)-Q_n(f)| 
\,\le\, \sum_{m\in\N_0^d}\, Z_n(m)^{1-1/{p'}} 
		\Big(\frac{M_{B_n,\supp(\Psi_m\ast f)}}{\det(B_n)}\Big)^{1-1/p}\,
		\|\Lambda_m\|_1 \,\|\Psi_m\ast f\|_p
\] 
By construction, the third factor is bounded by a constant. 
The second factor converges, as $n\to\infty$, to 
${\rm vol}_d(\supp(\Psi_m\ast f))^{1-1/p}\le{\rm vol}_d([-1,2]^d)^{1-1/p}$, 
see Remarks~\ref{rem:volume}~\&~\ref{up}. 
Hence, we obtain with Lemma~\ref{lem:Z} that
\begin{equation}\label{eq:error-limit-B}
|I(f)-Q_n(f)| \,\lesssim\, \sum_{\substack{m\in\N_0^d:\\ |m|_1>r_n}}\, 
(2^{|m|_1}/n)^{1/p}\; \|\Psi_m\ast f\|_p
\,=\, n^{-1/p} \sum_{\substack{m\in\N_0^d:\\ |m|_1>r_n}}\, 
2^{|m|_1(1/p-s)}\; 2^{s|m|_1} \|\Psi_m\ast f\|_p
\end{equation}
with $r_n=\log_2(n)-c$ from \eqref{eq:rn}.
Applying H\"older's inequality one more time, with $1/\theta+1/\theta'=1$, 
we finally obtain
\[\begin{split}
e(Q_n,\Bo) \,&\lesssim\, n^{-1/p} 
\Big(\sum_{m\colon |m|_1>r_n}\, 
2^{\theta'|m|_1(1/p-s)}\Big)^{1/\theta'} 
\,<\, n^{-1/p} 
\Big(\sum_{\ell>r_n}\, (\ell+1)^{d-1}
2^{\theta'\ell(1/p-s)}\Big)^{1/\theta'} \\
\,&\lesssim\, n^{-s} (\log n)^{(d-1)(1-1/\theta)},
\end{split}\]
since $s>1/p$. 
This proves Theorem~\ref{thm:besov}.

\subsection*{The result for $\mathring{\mathbf F}_{p,\theta}^s$}

We prove Theorem \ref{thm:TL}. If $s>\max\{1/p,1/\theta\}$ Theorem~\ref{thm:TL},(i) directly follows from Theorem \ref{thm:besov}. In fact, 
in case $p\leq \theta$ we have $\Fspt \hookrightarrow \Bspt$ for all $s$. If $p>\theta$ we use the embedding 
$\mathring{\mathbf F}_{p,\theta}^s \hookrightarrow \mathring{\mathbf B}^s_{\theta,\theta}$ due to the compact support and 
the definition via local means. To apply the results in Theorem \ref{thm:besov} we need $s>1/\theta$.
It remains to deal with the situation $p>\theta$ and $1/p < s \le 1/\theta$. 

Again, let $1/p+1/p'=1/\theta+1/\theta'=1$. We obtain from \eqref{eq:error} 
that
\begin{equation*}
 \begin{split}
|I(f)&-Q_n(f)| \,=\, \abs{\sum_{m\in\N_0^d} \l\sum_{k\neq0} \F\Lambda_m(B_n k)\, \e^{2\pi i k \cdot},
	\sum_{\ell\neq0}\F[\Psi_m\ast f](B_n \ell)\, \e^{2\pi i \ell \cdot} \r}\\
&=\, \abs{\sum_{m\in\N_0^d} \int_{[0,1]^d}\Big(
	\sum_{k\neq0} \F\Lambda_m(B_n k)\, \e^{2\pi i k x}\Big)\,
	\overline{\Big(\sum_{\ell\neq0}\F[\Psi_m\ast f](B_n \ell)\, \e^{2\pi i \ell x}\Big)} \dint x}
 \end{split}
\end{equation*}
Interchanging summation and integration and applying H\"older's inequality twice yields 
\begin{equation*}\label{eq:error-F}
\begin{split}
|I(f)-Q_n(f)|&\le\, \int_{[0,1]^d}\Big(\sum_{m\in\N_0^d} 2^{-\theta' s |m|_1} \Big|
	\sum_{k\neq0} \F\Lambda_m(B_n k)\, \e^{2\pi i k x}\Big|^{\theta'}\Big)^{1/\theta'}\,\\
	&\qquad\qquad\qquad
	\cdot\Big(\sum_{m\in\N_0^d} 2^{\theta s |m|_1} \Big|\sum_{\ell\neq0}\F[\Psi_m\ast f](B_n \ell)\, \e^{2\pi i \ell x}\Big|^\theta\Big)^{1/\theta} \dint x\\
&\le\, \Big\|\Big(\sum_{m\in\N_0^d} 2^{-\theta' s |m|_1} \Big|
	\sum_{k\neq0} \F\Lambda_m(B_n k)\, \e^{2\pi i k x}\Big|^{\theta'}\Big)^{1/\theta'}\Big\|_{L_{p'}([0,1]^d)} \\
	&\qquad\qquad\qquad
	\cdot\Big\|\Big(\sum_{m\in\N_0^d} 2^{\theta s |m|_1} \Big|\sum_{\ell\in\Z^d}\F[\Psi_m\ast f](B_n \ell)\, \e^{2\pi i \ell x}
	\Big|^\theta\Big)^{1/\theta} \Big\|_{L_{p}([0,1]^d)}\,. \\
\end{split}
\end{equation*}

\noindent At first, we bound the second factor. 
For this recall that the supports of $\Psi_m\ast f$ are subsets 
of $[-1,2]^d$ and set $M_n=M_{B_n,[-1,2]^d}$ in Lemma~\ref{lem:norm1}.
Hence, we obtain from Lemma~\ref{lem:norm1} (with $f_m=2^{s|m|_1}\Psi_m\ast f$) that
\[\begin{split}
&\Big\|\Big(\sum_{m\in\N_0^d} 2^{\theta s |m|_1} \Big|\sum_{\ell\in\Z^d}\F[\Psi_m\ast f](B_n \ell)\, \e^{2\pi i \ell \cdot}
	\Big|^\theta\Big)^{1/\theta} \Big\|_{L_{p}([0,1]^d)}\\
%&\qquad\qquad\qquad=\, \det(T_n)^p\,\Big\|\Big(\sum_{m\in\N_0^d} 2^{\theta s |m|_1} 
	%\Big|\sum_{\ell\in\Z^d}\Psi_m\ast f(T_n(\ell+\cdot)) 
	%\Big|^\theta\Big)^{1/\theta} \Big\|_{L_{p}([0,1]^d)}^p \\
%&\qquad\qquad\qquad=\, \det(T_n)^p\,M_n^p\int_{[0,1]^d}\Big(\sum_{m\in\N_0^d} 2^{\theta s |m|_1} 
	%\Big|\frac1{M_n}\sum_{\ell\in\Z^d}\Psi_m\ast f(T_n(\ell+x)) 
	%\Big|^\theta\Big)^{p/\theta} \dint x \\
%&\qquad\qquad\qquad\le\, \det(T_n)^p\,M_n^p\int_{[0,1]^d}\Big[\frac1{M_n}\sum_{\ell\in\Z^d}\Big(\sum_{m\in\N_0^d} 2^{\theta s |m|_1} 
	%\Big|\Psi_m\ast f\bigl(T_n(\ell+x)\bigr) 
	%\Big|^\theta\Big)\Big]^{p/\theta} \dint x \\
%&\qquad\qquad\qquad\le\, \det(T_n)^p\,M_n^{p-1}\sum_{\ell\in\Z^d}\int_{[0,1]^d}\Big(\sum_{m\in\N_0^d} 2^{\theta s |m|_1} 
	%\Big|\Psi_m\ast f\bigl(T_n(\ell+x)\bigr) 
	%\Big|^\theta\Big)^{p/\theta} \dint x \\
%&\qquad\qquad\qquad=\, \det(T_n)^p\,M_n^{p-1}\int_{\R^d}\Big(\sum_{m\in\N_0^d} 2^{\theta s |m|_1} 
	%\Big|\Psi_m\ast f\bigl(T_n x\bigr) 
	%\Big|^\theta\Big)^{p/\theta} \dint x \\
&\qquad\qquad\qquad\le\, \bigl(\det(T_n)\,M_n\bigr)^{1-1/p}\Big\|\Big(\sum_{m\in\N_0^d} 2^{\theta s |m|_1} 
	\Big|\Psi_m\ast f \Big|^\theta\Big)^{1/\theta} \Big\|_{p} 
\,\lesssim\, \|f\|_{F_{p,\theta}^s}.
\end{split}\]
%We used that $p>\theta \ge1$.
%
It remains to bound the first factor. 
For this let
\[
\wt\Lambda_{m,n}(x) \,:=\,\sum_{k\neq0} \F\Lambda_m(B_n k)\, \e^{2\pi i k x}
\]
and note that $p>\theta$ implies $p'<\theta'$. 
 
We prove the upper bound by splitting the sum 
into two parts. This approach was already used in \cite{Te93} to prove the 
result for the Fibonacci cubature rule in $\mathbf{W}_p^s(\mathbb{T}^2)$.
Let $L_n:=(d-1)\log\log n$ and bound the first factor from above by
\[\begin{split}
&\Big\|\Big(\sum_{m\colon |m|_1\le r_n+L_n} 2^{-\theta' s |m|_1} \Big|
	\wt\Lambda_{m,n}(\cdot)\Big|^{\theta'}
	\Big)^{1/\theta'}\Big\|_{L_{p'}([0,1]^d)} \\
&\qquad\qquad\qquad\qquad\qquad +\; \Big\|\Big(\sum_{m\colon |m|_1> r_n+L_n} 2^{-\theta' s |m|_1} \Big|
	\wt\Lambda_{m,n}(\cdot)\Big|^{\theta'}
	\Big)^{1/\theta'}\Big\|_{L_{p'}([0,1]^d)}. \\
\end{split}\]
Using $p'\le\theta'$ and $s<1/\theta$ we use Lemma~\ref{lem:norm2} and 
Lemma~\ref{lem:Z} 
to bound the first summand by
\[\begin{split}
\Big(\sum_{m\colon |m|_1\le r_n+L_n} 2^{-\theta' s |m|_1} 
	\Big\|\wt\Lambda_{m,n}\Big\|_{L_{\theta'}([0,1]^d)}^{\theta'}
	\Big)^{1/\theta'}
\,&\lesssim\, n^{-1/\theta} \Big(\sum_{\ell=r_n}^{r_n+L_n} (\ell+1)^{d-1}\, 
	2^{-\theta' \ell(s-1/\theta)} \Big)^{1/\theta'}\\
\,&\lesssim\, n^{-s}\, (\log n)^{(d-1)(1-s)}.
\end{split}\]
In the case $s=1/\theta$ this sum is bounded by
$n^{-s}\, (\log n)^{(d-1)(1-s)}\, (\log\log n)^{(1-s)}$.

To bound the second summand we replace the $\ell_{\theta'}$-norm inside by 
a $\ell_{p'}$-norm. We obtain for $s>1/p$ again by  Lemma~\ref{lem:norm2} and 
Lemma~\ref{lem:Z} the upper bound
\[\begin{split}
\Big(\sum_{m\colon |m|_1> r_n+L_n} 2^{-p' s |m|_1} 
	\Big\|\wt\Lambda_{m,n}\Big\|_{L_{p'}([0,1]^d)}^{p'}
	\Big)^{1/p'}
\,&\lesssim\, n^{-1/p} \Big(\sum_{\ell=r_n+L_n}^{\infty} (\ell+1)^{d-1}\, 
	2^{-p' \ell(s-1/p)} \Big)^{1/p'}\\
\,&\lesssim\, n^{-s}\, (\log n)^{(d-1)(1-s)}.
\end{split}\]
This finally proves Theorem~\ref{thm:TL}.

\medskip

\begin{rem}\label{rem:universal}
It is also possible to study a more general scale of function classes, 
where the mixed smoothness $s$ is replaced by a mixed smoothness vector $\bar{s}=(s_1,\dots,s_d)$. In a way, one
assumes a bounded mixed smoothness with different derivatives in different directions. Note, that this framework is
different from the classical ``anisotropic'' setting in \cite[Chapter IV.6]{Te93} and \cite[Sect.\ 4.3]{Nik75}.
Our definition of the spaces, see Definitions~\ref{def:besov}~\&~\ref{def:TL}, 
can be adopted to this setting by replacing the weights $2^{s|m|_1 \theta}$ 
there by $2^{(s_1 m_1+\dots+s_d m_d)\theta}$. We denote these spaces by 
$\T^{\bar s}_{p,\theta}$. With the standard technique from \cite[Lem.\ C]{D} in combination with our proof
technique in Subsection \ref{subsec:proofs} one can show for $1\leq p<\infty$, $1\leq \theta \leq \infty$ and
$s>\max\{1/p,1/\theta\}$ the relation
\[
e(Q_n, \mathring{\T}^{\bar s}_{p,\theta}) \;\asymp\; n^{-s_{\min}}\, (\log n)^{(\nu-1)(1-1/\theta)}
\]
where $s_{\text{min}}:=\min_i s_i$ and $\nu:=\#\{j: s_j=s_{\text{min}}\}$. Note, that there is no $d$-dependence in the
rate anymore. The corresponding results for Besov spaces are stated in~\cite{Du2}.
\end{rem}

%%%%%%%%%%%%%%%%%%%%%%%%%%%%%%%%%%%%%%%%%%%%%%%%%%%%%%%%%%%%%%%%%%%

\section{Quasi-Banach and limiting cases}\label{sec:quasi}

In this section we deal with the remaining cases of the Besov and 
Triebel-Lizorkin scales, that are not treated in the previous sections. 
We are interested in numerical integration and hence only in classes 
of continuous functions. Besides the quasi-Banach cases $\min\{p,\theta\}<1$ with 
$s>1/p$ we will also consider the limiting case $s=1/p$. 
In the latter case additional assumption are needed to assure continuity.

% We will state the results for $\Ao$. 
% But, however, note that all results of this section are proven by embedding into a space 
% with $p,\theta\ge1$, which also works for the spaces $\Aspt(\tor)$. 
% Moreover, we have the same upper bounds on 
% $e(\wt Q_n^\psi, \Aspt(\tor))$ as we have for $e(Q_n, \Ao)$, see 
% Theorems~\ref{thm:modi1-B} and \ref{thm:modi1-F}, in the cases $p,\theta\ge1$. 
% Hence, the results below could be stated also with 
% $e(Q_n, \Ao)$ replaced by $e(\wt Q_n^\psi, \Aspt(\tor))$.

\subsection{The situation $\min\{p, \theta\}<1$ and $s>1/p$} \label{subsec:quasi}

In this section we deal with the classes $\Bo$ and $\Fo$ with $p,\theta<1$, 
i.e.~the quasi-Banach cases. 
Here, $p<1$ affects the asymptotical error order negatively, 
while $\theta<1$ does not.
The presented lower bounds are given in the upcoming Section~\ref{sec:lower}. 

\begin{cor}\label{cor:small_t} Let $1\le p\le\infty$ (with $p<\infty$ in 
the \T-case), $0<\theta<1$ and $s>1/p$.\\
{\em (i)} Then
\[
e(Q_n, \Bo) \,\asymp\, n^{-s}.
\]
{\em (ii)} If $s\ge 1$ then 
\[
e(Q_n, \Fo) \,\asymp\, n^{-s}.
\]
{\em (iii)} If $1/p<s< 1$ then 
\[
e(Q_n, \Fo) \,\lesssim\, n^{-s}\, (\log n)^{(d-1)(1-s)}\,.
\]
\end{cor}

\bproof
The upper bounds follow from the embeddings 
$\A^s_{p,\theta}\hookrightarrow\A^s_{p,1}$ for $\theta<1$, 
together with Theorems~\ref{thm:besov}~\&~\ref{thm:TL}.\\
\eproof

\begin{cor}\label{cor:small_p} Let $0<p<1$, $0<\theta \leq  \infty$ and $s>1/p$.\\
{\em (i)} Then 
\[
e(Q_n, \Bo) \,\asymp\, n^{-s+1/p-1} (\log n)^{(d-1)(1-1/\theta)_+}\,,
\]
{\em (ii)} and 
\[
e(Q_n, \Fo) \,\asymp\, n^{-s+1/p-1}.
\]
\end{cor}

\bproof The stated bounds are a direct consequence  of 
Theorem~\ref{thm:besov} and Theorem~\ref{thm:TL}. 
In fact, from Lemma~\ref{emb} we know that we have for $p<1$ the embeddings 
\[
\Bo \hookrightarrow \mathring{\mathbf{B}}_{1,\theta}^{s+1-1/p}
\quad\text{ and }\quad 
\Fo \hookrightarrow  \mathring{\mathbf{F}}_{1,1}^{s+1-1/p} = \mathring{\mathbf{B}}_{1,1}^{s+1-1/p}. 
\]
\eproof

\begin{rem} In contrast to the case $p\ge1$, there is no dependence on 
$\theta$ in (ii). In particular, there is no effect of 
``small smoothness''.
\end{rem}

\subsection{Limiting cases}\label{subsec:limit}

In this section we deal with the limiting situation $s=1/p$. 
When $\theta \leq 1$ in the {\bf B}-case or $p\leq 1$ in the {\bf F}-case 
we have the continuous embedding $\Aspt(\R^d) \hookrightarrow C(\R^d)$. 

\begin{thm}\label{thm:limit}
{\em (i)} Let $0< p\le\infty$ and $0<\theta \leq 1$. Then
\[
e(Q_n, \mathring{\B}_{p,\theta}^{1/p}) \,\asymp\, n^{-1/\max\{p,1\}}.
\]
{\em (ii)} Let $0<p<1$ and $0<\theta\leq \infty$. Then 
\[
e(Q_n, \mathring{\mathbf{F}}_{p,\theta}^{1/p}) \,\asymp\, n^{-1}.
\]

\end{thm}

\bproof Due to the embeddings $\mathring{\mathbf{F}}_{p,\theta}^{1/p},\mathring{\B}_{p,\theta}^{1/p} \hookrightarrow
\mathring{\mathbf{F}}^1_{1,1} = \mathring{\mathbf{B}}^1_{1,1}$ if $p<1$ it suffices to prove (i) in case $1\leq p\leq
\infty$, $s=1/p$ and $\theta = 1$.
It follows the same line as the proof in the case 
$s>1/p$. In fact, \eqref{eq:error-limit-B} for $s=1/p$ shows 
\[
|I(f)-Q_n(f)| \,\lesssim\, n^{-1/p} \sum_{\substack{m\in\N_0^d:\\ |m|_1>r_n}}\, 
2^{|m|_1/p} \|\Psi_m\ast f\|_p 
\,\le\, n^{-1/p}\,\|f\|_{\B_{p,1}^{1/p}}.
\]
%The proof in the \T-case is more involved. 
%Recall that the functions $\Psi_m\in\S(\R^d)$ that were constructed in 
%Section~\ref{sec:spaces} are admissible kernels (in the sense of 
%Definition~\ref{def:kernels}) and satisfy 
%\[
%\supp(\Psi_m) \,=\, \{x\in\R^d\colon \}
%\]
\eproof

Unfortunately, our proof techniques do not seem to work for $p=1$ in (ii). The
following result, which is probably not sharp, follows from the 
embedding in Lemma~\ref{emb}(iii), together with Theorem \ref{thm:limit}(i). We
conjecture that $e(Q_n, \mathring{\T}_{1,\theta}^{1})\asymp n^{-1}(\log
n)^{(d-1)(1-1/\theta)}$. We leave this as an open problem.
\begin{cor} Let $0<\theta\leq \infty$. For any $\eps>0$ there is a constant
$c_{\eps}$ such that 
\[
  e(Q_n, \mathring{\T}_{1,\theta}^{1}) \lesssim c_{\eps}n^{-(1-\eps)}\,.
\]
\end{cor}
\bproof For an arbitrary $p>1$ we have the embedding
$\mathring{\T}^1_{1,\theta} \hookrightarrow \mathring{\B}^{1/p}_{p,1}$, 
see Lemma~\ref{emb}(iii). By Theorem \ref{thm:limit} we obtain the rate
$n^{-1/p}$.\\
\eproof

\section{Lower bounds} \label{sec:lower}

The lower bounds that we want to present are valid for \emph{arbitrary} 
cubature formulas. For this we study the quantity $\mbox{Int}_n(\mathbf{F}_d)$ 
from \eqref{eq:minimal} for the spaces $\mathbf{F}_d=\Aspt$.
%Moreover, since $\Ao$ is continuously embedded in $\Aspt(\mathbb{T}^d)$, it 
%is sufficient to present the lower bounds for the ``smallest'' class $\Ao$. 
There are already lower bounds 
for Besov classes in the literature see \cite{Te90} and the recent works \cite{DU14}, \cite{Te14}. None of those
references gives
lower bounds for $p<1$. We will provide them using an approach which is close to the one in \cite{DU14}. 
We will use the modern tool
of atomic decompositions \cite{Vyb06}, see Section~\ref{atomic} below, to construct appropriate fooling functions.

\subsection{Atomic decomposition}
\label{atomic}

We will describe the notion of an atom first. For $j\in \N_0^d$ and $k\in \Z^d$ let $Q_{j,k}$ denote the cube with center $(2^{-j_1}m_1,...,2^{-j_d}m_d)$ and with sides parallel to the coordinate axes of length $2^{-j_1},...,2^{-j_d}$\,. For $\gamma>0$ we denote with $\gamma Q_{j,k}$ the cube concentric with $Q_{j,k}$ with sides also parallel to the axes and length $\gamma 2^{-j_1},...,\gamma 2^{-j_d} $\,.

\begin{defi}\label{atom} Let $K \in \N_0$, $L+1 \in \N_0$ and $\gamma>1$. A $K$-times differentiable complex-valued function $a_{j,k}$ is called $(K,L)$-atom centered at $Q_{j,k}$ if \\
{\em (i)} $\supp\, a_{j,k} \subset \gamma Q_{j,k}\,,$\\
{\em (ii)} $|D^{\alpha}a_{j,k}(x)| \leq 2^{\alpha\cdot j}\quad,\quad$ for all $\alpha = (\alpha_1,...,\alpha_d)\in \N_0^d,~0\leq \alpha_i\leq K$\,,\\
{\em (iii)} and there are the coordinate-wise moment conditions
$$
    \int_{-\infty}^{\infty}x_i^{\ell} a_{j,k}(x) dx_i = 0 \quad\mbox{if}\quad i=1,...,d,~\ell=0,...,L,~j_i\geq 1\,.
$$
\end{defi}

\noindent The following Proposition is due to Vyb\'iral \cite{Vyb06}. Recall, that $\sigma_p:=\max\{0,1/p-1\}$ for $0<p<\infty$\,.

\begin{prop} \label{atomicdec}Let $0<p,\theta \leq \infty$ and $r\in \R$. Fix $K\in \N_0$ and $L+1 \in \N_0$ with 
$$
    K \geq (1+\lfloor r\rfloor)_+\quad \mbox{and}\quad L\geq \max\{-1,\lfloor \sigma_p - r\rfloor\}\,.
$$
If $\{\lambda_{j,k}\}_{j,k}$ is a sequence of complex-valued coefficients and $\{a_{j,k}(x)\}_{j,k}$ a collection of $(K,L)$-atoms centered at $Q_{j,k}$ then the function  
\begin{equation}\label{tf}
	f:=\sum\limits_{j\in \N_0^d} \sum\limits_{k\in \Z^d} \lambda_{j,k}a_{j,k}(x)
\end{equation}
belongs to $\Bspt$ if the right-hand side in \eqref{flelambda} below is finite. Then it holds 
\begin{equation}\label{flelambda}
      \|f\|_{\Bspt} \lesssim \Big(\sum\limits_{j\in \N_0^d} 2^{|j|_1(r-1/p)\theta}\Big[\sum\limits_{k\in \Z^d}|\lambda_{j,k}|^{p}\Big]^{\theta/p}\Big)^{1/\theta}\,.
\end{equation}

\end{prop}

\noindent Consider a function $\varphi \in C_0^\infty(\R)$ with support in $[0,1]$. Now we define
$$
   a_{j,k}(x_1,...,x_d) := \varphi(2^{j_1}x_1-k_1)\cdot...\cdot \varphi(2^{j_d}x_d-k_d)\quad,\quad x\in \R^d,\,j\in \N_0^d,\, k\in \Z^d
$$
and observe that $a_{j,k}$ is $(K,-1)$-atom for any $K\in \N_0$ centered at $Q_{j,k}$ with $\gamma = 1$ according to
Definition \ref{atom}. Note, that we do not have moment conditions. According to Proposition \ref{atomicdec} we do not
need moment conditions if $r>\sigma_p$.

\subsection{Test functions}
Now we are in a position to define test functions of type \eqref{tf} in order to prove the Theorem below. 
By \eqref{flelambda} we are able to control the norm $\|\cdot\|_{\Bspt}$. Following \cite[Thm.\ 4.1]{DU14} we will use test functions of type
\begin{equation}\label{g_1}
g^1_{s,\theta} := C2^{-sm}m^{-(d-1)/\theta} \sum\limits_{|j|_1=m+1}\sum\limits_{k \in K_j(X_n)} a_{j,k}\,,
\end{equation}
where $D_\ell := \{1,...,2^{\ell}-1\}$, $\ell\in \N$, and 
$K_j(X_n) \subset \mathbb{D}_j := D_{j_1} \times ...\times D_{j_d}$ 
depends on the set of integration nodes $X_n:=\{x^1,...,x^n\}$\,. 

Let $\Ao \cap C(\R^d)$ denote the function class  $\Ao \cap C(\R^d)$ equipped with the norm $\|\cdot\|_{\Aspt}$\,.

\begin{thm}\label{thm:lb} Let $0<p,\theta \leq \infty$ and $s>\sigma_p:=\max\{0,1/p-1\}$.\\ 
{\em (i)} Then
$$
    \mbox{Int}_n(\Bo \cap C(\R^d)) \gtrsim n^{-s+(1/p-1)_+}(\log n)^{(d-1)(1-1/\theta)_+}\,.
$$
{\em (ii)} If $1\leq p <\infty$ then 
$$
    \mbox{Int}_n(\Fo \cap C(\R^d)) \gtrsim n^{-s}(\log n)^{(d-1)(1-1/\theta)_+}\,.
$$
{\em (iii)} If $0<p<1$ then 
$$
    \mbox{Int}_n(\Fo \cap C(\R^d)) \gtrsim n^{-s+1/p-1}\,.
$$
\end{thm}

\bproof We only need to prove (i). In fact, (ii) follows from (i) and the embeddings 
$\ensuremath{\mathring{\mathbf B}_{\theta,\theta}^s} = \mathring{\mathbf F}_{\theta,\theta}^s \hookrightarrow \Fo$ 
if $\theta \geq p$ (H\"older's inequality)
and $\Bo \hookrightarrow \Fo$ if $\theta < p$ (Lemma~\ref{emb}). The relation in (iii) follows from (i) and the embedding
$\ensuremath{\mathring{\mathbf B}_{p,\min\{p,\theta\}}^s} \hookrightarrow \Fo$ where $\min\{p,\theta\}<1$\,.\\
Let us prove (i): We follow the arguments in \cite[Thm.\ 4.1]{DU14}. Let $n$ be given and $X_n = \{x^1,...,x^n\} \subset [0,1]^d$ be an 
arbitrary set of $n$ points. Without loss of generality we assume that $n = 2^m$. Since $Q_{j,k} \cap Q_{j,k'} = \emptyset$ for $k\neq k'$ we have for every $|j|_1 = m+1$ a set $\mathbb{D}_j(X_n) \subset \mathbb{D}_j$ 
with $\#\mathbb{D}_j(X_n) \gtrsim 2^{|j|_1}$ and $X_n \cap \mathbb{D}_j(X_n) = \emptyset$.

{\em (a)} Let $p,\theta \geq 1$. We choose the test function \eqref{g_1} with $K_j(X_n) := \mathbb{D}_j(X_n)$. Note, that for every $k\in K_j(X_n)$ we have that 
$\supp\, a_{j,k} \subset [0,1]^d$. The function $g^1_{s,\theta}$ in \eqref{g_1} is defined via a finite sum of continuous functions with support contained in $[0,1]^d$. Therefore, $g^1_{s,\theta}$ is continuous and $\supp\, g^1_{s,\theta} \subset [0,1]^d$. By Proposition \ref{atomicdec} and \eqref{flelambda} we can arrange $C>0$
such that $\|g^1_{s,\theta}\|_{\Bspt} \leq 1$\,. Clearly, $g^1_{s,\theta}$ belongs to $\Bo \cap C(\R^d)$. It is obvious that 
$$
   \int_{[0,1]^d} g^1_{s,\theta}\,dx \asymp 2^{-ms}m^{(d-1)(1-1/\theta)}\,.
$$
Of course, a cubature rule admitted in \eqref{eq:minimal} that uses the points $X_n$ produces a zero output. This proves (i) in case $p,\theta \geq 1$.

{\em (b)} Let $p\geq 1$ and $\theta <1$. Let us choose a $j\in \N^d$ with $|j|_1 = m+1$ and define the function
\begin{equation} \label{g_2}
  g^2_{s}:=C2^{-sm}\sum\limits_{k\in \mathbb{D}_j(X_n)} a_{j,k}.
\end{equation}
Again, by Proposition \ref{atomicdec} there is a $C>0$ such that $\|g^2_s\|_{\Bspt}\leq 1$. Moreover, we have 
$$
    \int_{[0,1]^d} g^2_{s}\,dx \asymp 2^{-sm}\,.
$$
With the same reasoning as in (a) this proves (i) in case $\theta <1$.

{\em (c)} Let $p<1$ and $\theta\geq 1$. We define the test function 
$$
    g^3_{s,p,\theta}:= C2^{-sm}m^{-(d-1)/{\theta}}2^{m/p}\sum\limits_{|j|_1 = m+1}a_{j,k_j}\,,
$$
where $k_j$ is chosen from $\mathbb{D}_{j}(X_n)$. By \eqref{flelambda} we find a $C>0$ such that $\|g^3_{s,\theta}\|_{\Bspt} \leq 1$. Computing the integral gives 
$$
    \int_{[0,1]^d} g^3_{s,p,\theta}\,dx \asymp 2^{(-s+1/p-1)m}m^{(d-1)(1-1/\theta)}\,.
$$
With the same reasoning as in (a) and (b) this proves (i) in case $p<1$ and $\theta\geq 1$.

{\em (d)} Finally, if $0<p,\theta<1$ we simply take one single atom 
$g^4_{s,p}:=2^{-ms}2^{m/p}a_{j,k}$ with $j \in \N^d$ and $k \in \mathbb{D}_j(X_n)$. 
Again, there is a $C>0$ such that $\|g^4_{s,p}\|_{\Bspt} \leq 1$ and 
$$
    \int_{[0,1]^d} g^4_{s,p}\,dx \asymp 2^{(-s+1/p-1)m}\,,
$$
which finishes the proof. \eproof

\begin{rem} Note that the lower bound in (ii) differs from our upper bounds in the case of ``small smoothness'' $p>\theta$ and 
$1/p<s\le \min\{1,1/\theta\}$. 
Hence, to prove optimality of Frolov's cubature formula for each $\Ao$ 
it remains to prove the corresponding lower bound for ``small'' smoothness 
in Triebel-Lizorkin and Sobolev spaces. 
This seems to be very delicate and we leave it as an open problem. 
However, we conjecture that our upper bounds are tight. This is supported by the fact 
that for $d=2$ the corresponding lower bound was proven for the Fibonacci 
cubature formula, see~\cite[Theorem~2.5]{Te93}, 
which is conjectured to be optimal, cf.~\cite{HO14}.
\end{rem}

%%%%%%%%%%%%%%%%%%%%%%%%%%%%%%%%%%%%%%%%%%%%%%%%%%%%%%%%%%%%%%%%%%%%%%%%%%

\bigskip
\noindent
{\bf Acknowledgement.} The authors would like to thank the organizers of the
conference ``Monte Carlo and Quasi-Monte Carlo Methods in Scientific
Computing'' (Leuven, 2014), where this work has been initiated, for providing a
pleasant and fruitful working atmosphere. They would further like to thank
Dauren Bazarkhanov for pointing out the reference \cite{Du2}, and 
Glenn Byrenheid, Aicke Hinrichs, Dinh D\~ung, Erich Novak, Jens Oettershagen,
Winfried Sickel and Vladimir N. Temlyakov for several helpful remarks and
comments on earlier versions of this manuscript. The work of Tino Ullrich is supported by the
Hausdorff-Center for Mathematics, 
University of Bonn, and the DFG-Emmy-Noether programme UL403/1-1. 

\goodbreak

\end{document}